\documentclass[1p]{preprint}

\journal{}

\usepackage[utf8]{inputenc}
\usepackage{snapshot}

\usepackage{etex}
\reserveinserts{28}

\usepackage{latexsym}
\usepackage{subfigure}
\usepackage{times}
\usepackage{enumerate,amssymb,euscript,amsmath}
\usepackage{pst-all}
\usepackage{graphicx}
\usepackage{pgfplots}
\usepackage{subfigure}
\usepackage{pdfsync}
\usepackage[draft,inline,marginclue]{fixme}
\usepackage{hyperref}
\usepackage{url}

\usepackage{colortbl}
\usepackage{pgfplots}
\usepackage{pgfplotstable}
\usepackage{tikz}
\usepackage{tkz-euclide}
\usetkzobj{all}
\usetikzlibrary{decorations.pathreplacing,shapes.misc,calc,intersections,arrows,external}

\usepackage{ifem-luca-nella-private}

\graphicspath{{./figures/}}

\FXRegisterAuthor{lh}{alh}{LH}
\FXRegisterAuthor{nr}{anr}{NR}

\newtheorem{theo}{Theorem}

\newtheorem{lemma}[theo]{Lemma}
\newtheorem{problem}{Problem}
\newtheorem{remark}{Remark}

\def\wbox#1;#2;{\vbox{\hrule\hbox{\vrule height#1mm\kern#2mm\vrule
  height#1mm}\hrule}}

\def\Dim{\noindent {\it Proof.}\hskip.2truecm}
\def\fineDim{\nobreak\hfill\wbox2.3;2.3;\linebreak}
\newenvironment{proof}{\Dim}{\fineDim}

\begin{document}

\maketitle

\begin{frontmatter}

  \title{Error estimates in weighted Sobolev norms for \\
    finite element immersed interface methods}

  \author[sissa]{Luca Heltai} 
  \ead{luca.heltai@sissa.it}
  \author[wias]{Nella Rotundo\corref{cor}}
  \ead{nella.rotundo@wias-berlin.de}
  \cortext[cor]{Corresponding author. Tel: +49 30 20372-398.}
\address[sissa]{SISSA-International School for Advanced Studies\\
  via Bonomea 265, 34136 Trieste - Italy}
\address[wias]{Weierstrass Institute for Applied Analysis and Stochastics\\
  Mohrenstra\ss e 39, 10117 Berlin - Germany}

\begin{abstract} 

  When solving elliptic partial differential equations in a region containing immersed interfaces (possibly evolving in time), it is often desirable to approximate the problem using an independent background discretisation, not aligned with the interface itself. Optimal convergence rates are possible if the discretisation scheme is enriched by allowing the discrete solution to have jumps aligned with the surface, at the cost of a higher complexity in the implementation.

  A much simpler way to reformulate immersed interface problems consists in replacing the interface by a singular force field that produces the desired interface conditions, as done in immersed boundary methods. These methods are known to have inferior convergence properties, depending on the global regularity of the solution across the interface, when compared to enriched methods.

  In this work we prove that this detrimental effect on the convergence properties of the approximate solution is only a local phenomenon, restricted to a small neighbourhood of the interface. In particular we show that optimal approximations can be constructed in a natural and inexpensive way, simply by reformulating the problem in a distributionally consistent way, and by resorting to weighted norms when computing the global error of the approximation. 

\end{abstract}


\end{frontmatter}




\pagebreak

\section{Introduction}
\label{sec:intro}

Interface problems are ubiquitous in nature, and they often involve changes in topology or complex coupling across the interface itself. Such problems are typically governed by elliptic partial differential equations (PDEs) defined on separate domains and coupled together with interface conditions in the form of jumps in the solution and flux across the interface. 

It is a general opinion that reliable numerical solutions to interface problems can be obtained using body fitted meshes (possibly evolving in time), as in the Arbitrary Lagrangian Eulerian (ALE) framework~\cite{Hirt1974,DoneaGiulianiHalleux-1982-a}. However, in the presence of topological changes, large domain deformations, or freely moving interfaces, these methods may require re-meshing, or even be impractical to use.

Several alternative approaches exist that reformulate the problem using a fixed background mesh, removing the requirement that the position of the interface be aligned with the mesh. These methods originate from the Immersed Boundary Method (IBM), originally introduced by Peskin in~\cite{Peskin1972}, to study the blood flow around heart valves (see also~\cite{Peskin2002}, or the review~\cite{Mittal2005b}), and evolved into a large variety of methods and algorithms. 

We distinguish between two main different families of immersed methods. In the first family  the interface conditions are incorporated into the finite difference scheme, by modifying the differential operators,  or in the finite element space by enriching locally the basis functions to allow for the correct jump conditions in the gradients or in the solution. The second family leaves the discretisation intact, and reformulates the jump conditions using singular source terms.

Important examples of the first family of methods are given by the Immersed Interface Method (IIM)~\cite{Leveque1994} and its finite element variant~\cite{Li1998}, or the eXtended Finite Element Method (X-FEM)~\cite{Sukumar2000}, that exploits partition of unity principles~\cite{Melenk1996} (see also~\cite{HansboHansbo-2002-a, HansboHansbo-2004-a}). For a comparison between IIM and X-FEM see, for example,~\cite{Vaughan2006}, while for some details on the finite element formulation of the IIM see~\cite{Li2003a, Gong2008, Hou2013, Mu2013}.

The original Immersed Boundary Method~\cite{Peskin1972} and its variants belong to the second category. Singular source terms are formally written in terms of the Dirac delta distribution, and their discretisation follow two possible routes: i) the Dirac delta distribution is approximated through a smooth function, or ii) the variational definition of the Dirac distribution is used directly in the Finite Element formulation of the problem. For finite differences, the first solution is the only viable option, even though the use of smooth kernels may excessively smear the singularities, leading to large errors in the approximation~\cite{Hosseini2014}. In finite elements, instead, both solutions are possible. The methods derived from the Immersed Finite Element Method (IFEM) still use approximations of the Dirac delta distribution through the Reproducing Kernel Particle Method (RKPM)~\cite{ZhangGerstenberger-2004-Immersed-finite-0}. 

Variational formulations of the IBM were introduced in~\cite{BoffiGastaldi-2003-a, BoffiGastaldi-2007-Numerical-stability-0, BoffiGastaldiHeltaiPeskin-2008-a, Heltai-2008-a}, and later generalised in~\cite{Heltai2012b}, where the need to approximate Dirac delta distributions is removed by exploiting directly the weak formulation. Such formulations allow the solution of PDEs with jumps in the gradients without enriching the finite element space, and without introducing approximations of the Dirac delta distribution.

When IBM-like formulations are used to approximate interface problems, a natural deterioration is observed in the convergence of the approximate solution, which is no longer globally smooth, and cannot be expected to converge optimally to the exact solution (see the results section of~\cite{BoffiGastaldiHeltaiPeskin-2008-a}, or~\cite{Ramiere2008}). A formal optimal convergence can be observed in special cases~\cite{Lai2000}, but in general the global convergence properties of these methods is worse when compared to methods where the interface is captured accurately, either by local enrichment of the finite dimensional space, as in the X-FEM or IIM, or by using interface-fitted meshes, as in the ALE method.

In this work we show that this is only partially true, and that \emph{optimal} approximations can be constructed also when non-body fitted meshes are used, and when no explicit treatment of the jump conditions are imposed in the solution. This can be achieved in a natural and inexpensive way, simply by reformulating the problem in a distributionally consistent way, and by resorting to weighted norms when computing the global error of the approximation. We show here that the deterioration of the error is a purely local phenomena, restricted to a small neighbourhood of the interface itself. In particular we prove that by using suitable powers of the distance function from the interface  as weights in  weighted Sobolev norms when computing the errors, optimal error estimates can be attained globally.

Weighted Sobolev spaces~\cite{Kufner1985,Turesson2000}, provide a natural framework for the study of the convergence properties of problems with singular sources (see, for example,~\cite{Agnelli2014}).  These spaces are commonly used in studying problems with singularities in the domain (for example in axisymmetric domains~\cite{Belhachmi2006}, or in domains with external cusps~\cite{Duran2009a}) and when the singularities are caused by degenerate or singular behavior of the coefficients of the differential operator~\cite{Fabes1982a, Caffarelli2007, Cabre2015}. A particularly useful class of weighted Sobolev spaces is given by those spaces whose weights belong to the so-called Muckenhoupt class $A_p$~\cite{Muckenhoupt1972}. An extensive approximation theory for weighted Sobolev spaces is presented in~\cite{Nochetto2016}.

The ideas we present in this work are inspired by the works of
D'Angelo and Quarteroni~\cite{DAngelo2008} and
D'Angelo~\cite{Dangelo2012}, where the authors discuss the coupling
between one dimensional source terms and three dimensional
diffusion-reaction equations~\cite{DAngelo2008}, and finite element
approximations of elliptic problem with Dirac
measures~\cite{Dangelo2012}.

A general setting for the numerical approximation of elliptic problems
with singular sources is available in \cite{Drelichman2018}
and~\cite{Otarola2018}.

By applying the same principles, we recover optimal error estimates in
the approximation of problems with singular sources distributed along
co-dimension one surfaces, such as those arising in the variational
formulation of immersed methods.

In Section~\ref{sec:immers-bound-meth} and~\ref{sec:weight-sobol-spac}
we outline the problem we wish to solve, and introduce weighted
Sobolev spaces. Section~\ref{sec:numer-appr} is dedicated to the
definition of the numerical approximation, and to the proof of the
optimal convergence rates in weighted Sobolev norms.
Section~\ref{sec:validation} presents a numerical validation using
both two- and three-dimensional examples, while
Section~\ref{sec:conclusions} provides some conclusions and
perspectives.

\section{Model interface problem}
\label{sec:immers-bound-meth}

When approximating problems with interfaces using non-matching grids, one can choose among several possibilities. For example, one could decide to discretise the differential operators using finite differences, or to pose the problem in a finite dimensional space and leave the differential operators untouched,  as in the finite element case.  

In both cases, if one wants to enforce strongly the interface conditions, it is necessary to modify the differential operators (as in the IIM method~\cite{Leveque1994} for finite differences) or to enrich the finite dimensional space (as in the X-FEM~\cite{Sukumar2000}). A third option, that we will call \emph{distributional} approach, consists in leaving the space and the differential operators untouched, and to rewrite the jump conditions in terms of singular sources, as in the original Immersed Boundary Method~\cite{Peskin2002}.

The \emph{distributional} approach can be tackled numerically either by mollification of Dirac delta distributions, or by applying variational formulations, where the action of the Dirac distributions is applied using its definition to the finite element test functions. 

\begin{figure}[h]
\tikzsetnextfilename{gamma-only}
\centering
\begin{tikzpicture}

  \node at (1,1) (G){};
  \node at (2,4) (N){};
  \node at (1.5,4.5) (Np){};
  
  \draw [red] plot [smooth cycle, tension=1]
  coordinates { (G) (3,1)  (4,2)  (3,3) (N)};

  \node [below left] at (G) {$\Gamma$};
  \node [above] at (N) {$\nu$};
  
  \draw [-stealth] (N.south east) -- (Np);
  \fill (N) circle (1pt);

\end{tikzpicture}
\caption{An immersed interface in $\Re^n$}
\label{fig:interface}
\end{figure}
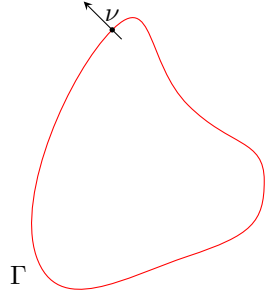

To fix the ideas, consider a Lipschitz, closed, interface $\Gamma \subset \Re^n$ of co-dimension one (as in Figure~\ref{fig:interface}), and the following model problem:

\begin{problem}[Interface only]
\label{pb:interface-only}
Given a function $f\in H^{-s}(\Gamma)$, $s\in \left(0,\frac12\right]$, find  a harmonic function $p$ in $\Re^n\setminus\Gamma$, such that:
\begin{equation}
\label{eq:auxiliary_p}
\begin{aligned}
& -\Delta p = 0 && \text{ in } \Re^n \setminus \Gamma, \\
& \jump{{p_\nu}} = f && \text{ on } \Gamma,\\
& \jump{p} = 0 && \text{ on } \Gamma.
\end{aligned}
\end{equation}
\end{problem}

The notation $\jump{\cdot}$ is used to indicate the jump across the interface $\Gamma$, and $p_\nu$ indicates the normal derivative of $p$, i.e., $\nu\cdot\nabla p$. The direction $\nu$ on $\Gamma$ is used to define precisely the meaning of $\jump{a}$ for any quantity $a$, i.e.:
\begin{equation}
  \label{eq:jump-definition}
  \jump{a} := a^{+} - a^{-},
\end{equation}
where $a^{+}$ lies on the same side of $\nu$.

Problem~\ref{pb:interface-only} admits a solution that can be constructed explicitly in terms of the  \emph{boundary integral representation} for harmonic functions (see, e.g.,~\cite{HsiaoWendland-2008-a}):
\begin{equation}
\label{eq:boundary_integral_representation_p}
  p(x) =  \int_\Gamma G(x-y) ~ f(y) \d \Gamma_y \quad \forall x \in \Re^n\setminus \Gamma,
\end{equation}
where $G$ is the fundamental solution of the Poisson problem in $\Re^n$:
\begin{equation}
\label{eq:green}
G(r) := 
\begin{cases}
-\frac{1}{2\pi} \log|r| & \text{ when }\, n = 2,
\\
\frac{1}{4\pi|r|}
& \text{ when }\, n = 3.
\end{cases}
\end{equation}

The function $G$ satisfies, in the distributional sense,
\begin{equation}
  \label{eq:laplacian-G}
  -\Delta_x G(x-y) = \delta(x-y), \qquad \forall x \text{ in }\Re^n\setminus\{y\}, 
\end{equation}
where $\delta$ is the $n$-dimensional \emph{Dirac delta distribution}, i.e., the distribution such that
\begin{equation}
  \label{eq:definition-dirac-delta}
  \int_{\Re^n} \delta(x-y) v(x) \d x := v(y), \qquad \forall v \in \D(\Re^n), \forall y \in \Re^n,
\end{equation}
and $\D(\Re^n)$ is the space of infinitely differentiable
functions, with compact support on $\Re^n$.

If $f$ is at least $H^{-\frac12}(\Gamma)$, then $p$ is globally in $H^1_{\text{loc}}(\Re^n)$, it is harmonic in the entire $\Re^n\setminus\Gamma$ (see, e.g.,~\cite{HsiaoWendland-2008-a} for a proof), and we can take its Laplacian in the entire $\Re^n$ in the sense of distributions.

Exploiting the boundary integral
representation~\eqref{eq:boundary_integral_representation_p}, the distributional laplacian of $p$ 
\begin{equation}
\label{eq:laplacian of bie p}
\begin{split}
-\Delta_x p(x) = & - \Delta_x \int_\Gamma G(x-y) ~ f(y) \d \Gamma_y \\
= & \int_\Gamma (-\Delta_x G(x-y) )  ~ f(y) \d \Gamma_y \\
= & \int_\Gamma \delta(x-y)   ~ f(y) \d \Gamma_y 
\quad \forall x \in \Re^n\setminus \Gamma,
\end{split}
\end{equation}
can be formally expressed in terms of a distributional operator $\LD$ as 
\begin{equation}
  \label{eq:definition M}
  (  \LD   f)(x) := \int_\Gamma \delta(x-y) ~ f(y)\d \Gamma_y.
\end{equation}

Notice that in the definition of the operator $\LD f$, the Dirac delta distribution $\delta$ is  defined through its action on functions in $\D(\Re^n)$, by the distributional definition (\ref{eq:definition-dirac-delta}). In $\LD f$, the Dirac distribution is convoluted with $f$ on a domain $\Gamma$ of co-dimension one with respect to $\Re^n$. The resulting distribution is zero everywhere, and singular \emph{across} $\Gamma$. 

This is usual in the Immersed Boundary and Immersed Interface literature, and should be interpreted as the distributional operator whose effect is to take the trace of the test function on $\Gamma$, and apply the duality product on $\Gamma$ with the function $f$:
\begin{equation}
  \label{eq:distributional-equivalence-of-dirac-on-Gamma}
  \langle \LD f, \varphi \rangle = \int_{\Re^n} \varphi(x)   \int_\Gamma \delta(x-y)   ~ f(y) \d
  \Gamma_y \d x := \int_\Gamma \varphi(y)  ~ f(y)  \d
  \Gamma_y  \qquad \forall \varphi \in \D(\Re^n).
\end{equation}

Similarly, let us consider a simply connected and convex domain  $\Omega$ of $\Re^n$, $n=2,3$, with Lipschitz boundary $\partial \Omega$, separated into $\Omega^+$ and $\Omega^-$ by  the surface $\Gamma$, as in Figure~\ref{fig:domain}, where we assume that  $\Gamma \cap \partial\Omega = \emptyset$.

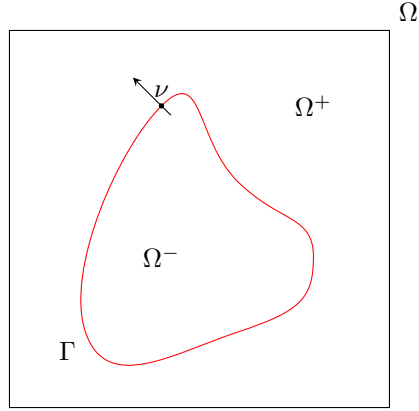
\begin{figure}[h]
\tikzsetnextfilename{domain}
\centering
\begin{tikzpicture}

  \draw (0,0) rectangle (5,5) node[anchor=south west](O) {$\Omega$};
  
  \node at (1,1) (G){};
  \node at (2,4) (N){};
  \node at (1.5,4.5) (Np){};
  
  \draw [red] plot [smooth cycle, tension=1]
  coordinates { (G) (3,1)  (4,2)  (3,3) (N)};

  \node at (4,4) {$\Omega^+$};
  \node at (2,2) {$\Omega^-$};
  \node [below left] at (G) {$\Gamma$};
  \node [above] at (N) {$\nu$};
  
  \draw [-stealth] (N.south east) -- (Np);
  \fill (N) circle (1pt);

\end{tikzpicture}
\caption{Domain representation}
\label{fig:domain}
\end{figure}

We use standard notations for Sobolev spaces  (see, for example,~\cite{Adams2003}), i.e.,  $H^s(A) = W^{s,2}(A)$, for real $s$, where $L^2(A) = H^0(A)$ and $H^1_0(A)$ represents square integrable functions on $A$ with square integrable first derivatives and whose trace is zero on the boundary $\partial A$ of the domain of definition $A$. 

\begin{problem}[Strong]
\label{pb:continuous}
Given $b\in L^2(\Omega)$ and $f\in H^{-s}(\Gamma)$, $s\in\left(0,\frac12\right]$, find a solution $u$ of the problem

\begin{equation}
\label{eq:problem_u}
\begin{aligned}
& -\Delta u = b && \text{ in } \Omega \setminus \Gamma, \\
& \jump{{u_\nu}} = f && \text{ on } \Gamma, \\
& \jump{{u}} = 0 && \text{ on } \Gamma, \\
& u = 0 && \text{ on } \partial\Omega.
\end{aligned}
\end{equation}
\end{problem}

Using the solution $p$ to Problem~\ref{pb:interface-only}, a solution to Problem~\ref{pb:continuous} can be written as $u=z+p$ where $z$ satisfies
\begin{equation}
\label{eq:auxiliary_z}
\begin{aligned}
& -\Delta z = b && \text{ in } \Omega, \\
& z = -p && \text{ on } \partial\Omega.
\end{aligned}
\end{equation}

Since $\Gamma \cap \partial\Omega = \{\emptyset\}$, and away from
$\Gamma$ the solution $p$ is harmonic, its restriction on
$\partial \Omega$ is at least Lipschitz and continous, and we conclude
that problem~(\ref{eq:auxiliary_z}) admits a unique solution which is
at least $H^2(\Omega)$, thanks to the assumption that $\Omega$ is
convex. The solution $u = z+p$ to Problem~\ref{pb:continuous} is
globally $H^1(\Omega)$, and at least $H^2(\Omega \setminus \Gamma)$.

Problem~\ref{pb:continuous} is equivalent to the following distributional formulation, where the jump conditions are incorporated into a singular right hand side.

\begin{problem}[Distributional]
\label{pb:distributional}
Given $b\in L^2(\Omega)$, $f\in H^{-s}(\Gamma)$, $s\in\left(0,\frac12\right]$, find the distribution $u$ such that
\begin{equation}
\label{eq:problem_dirac}
\begin{aligned}
  & -\Delta u = b   +   \LD  f && \text{ in } \Omega, \\
  & (  \LD   f)(x) := \phantom{-}\int_\Gamma \delta(x-y) ~ f(y)\d \Gamma_y,\\
  & u = 0 && \text{ on } \partial\Omega.
\end{aligned}
\end{equation}
\end{problem}

The function $u = z+p$ is then a solution to Problem~\ref{pb:continuous} in $\Omega\setminus\Gamma$, and to Problem~\ref{pb:distributional}, in the distributional sense, in the entire domain $\Omega$. The distributional definition of $ \LD f$ is derived in equations~(\ref{eq:green}--\ref{eq:laplacian of bie p}), and it is given by \begin{equation} \label{eq:distribution-S-D-applied-on-v} \begin{aligned} \langle \LD f, \varphi \rangle := & \phantom{-} \int_\Gamma f ~ \varphi \d \Gamma && \forall \varphi \in \D(\Omega),  \end{aligned} \end{equation}
where we exploit that $\D(\Omega) \subset  \D(\Re^n)$.

This formulation is at the base of the Immersed Boundary
Method~\cite{Peskin2002}. In the IBM, a problem similar to
Problem~\ref{pb:distributional} is discretised by Finite Differences,
and the Dirac delta distribution is replaced by a regularised delta
function, used to interpolate between non-matching sample points on
the co-dimension one surface $\Gamma$ and on the domain $\Omega$. 

When regularised Dirac distributions are used, a natural deterioration is observed in the convergence of the approximate solution. A formal second order convergence can be observed in special cases~\cite{Lai2000}, but in general the global convergence properties of these methods is worse when compared to methods where the interface is taken into account explicitly, like in the IIM~\cite{Leveque1994}. Moreover, the regularisation itself introduces an additional source of approximation, which smears out the singularity, and may deteriorate further the convergence properties of the method~\cite{Hosseini2014}.

\section{Variational formulation}
\label{sec:vari-form}

In order to study the regularity and the well posedeness of
Problem~\ref{pb:distributional}, we begin by showing that $\LD f$
belongs to $H^{-1}(\Omega)$, and therefore there exists a solution $u$
to the distributional Problem~\ref{pb:distributional} which is globally in
$H^1_0(\Omega)$.

We begin by recalling standard results for trace operators. For a bounded domain $\Omega^{+}$ (or $\Omega^{-}$) with (part of the) boundary $\Gamma$, given a function $u \in C^0 (\overline{\Omega^+} )$ (or $u \in C^0(\overline{\Omega^-})$), it makes sense to define the restriction of $u$ on $\Gamma$, simply by considering its pointwise restriction. For Sobolev spaces, we recall the following classical result (see, e.g.,~\cite[Theorem 3.38]{mclean2000strongly}, or \cite{ciarlet78}):

\begin{theo}[Trace theorem] 
\label{theo:trace}
Let $\Gamma$ be  a Lipschitz closed co-dimension one surface, splitting $\Omega$ into 
$\Omega^+$ and $\Omega^-$. 
For $0 < s < 1$ the interior and the exterior trace operators 
\[
\begin{split}
\gamma^{\mathrm{int}}:& H^{s+\frac12}({\Omega}^-)\to H^{s}(\Gamma),\\
\gamma^{\mathrm{ext}}:& H^{s+\frac12}({\Omega}^+)\to H^{s}(\Gamma),
\end{split}
\]
are bounded, linear, and injective mappings, that posses bounded right inverses. If the function $v$ is globally $H^{s+\frac12}(\Omega)$, then $\gamma^{\mathrm{int}} v  =  \gamma^{\mathrm{ext}} v$, and we omit the symbol altogether,  simply indicating with $v$ both the function in $H^{s+\frac12}(\Omega)$ and its restriction to $H^{s+\frac12}(\Gamma)$. 
There exists a constant $C_T>0$ such that   
\begin{equation}
\label{eq:trace_theorem}
   \| v \|_{s ,\Gamma} \leq C_T  \| v \|_{s+\frac12, \Omega}, 
\qquad \forall  v \in H^{s}(\Omega), 
   \qquad 0 < s < 1.
\end{equation}
\end{theo}

\begin{theo}[Regularity of $\LD$]
  \label{theo:regularity M}
  For $\Gamma$ Lipschitz and $s\in \left(0,\frac12\right]$, 
   the operator $\LD : H^{-s}(\Gamma) \to H^{-s-\frac12}(\Omega)$ defined as
  \begin{equation*}
    {}_{H^{-s-\frac12}(\Omega)}\langle\LD f, v\rangle_{H^{s+\frac12}_0(\Omega)} := \int_\Gamma f v \d \Gamma \qquad \forall v \in H_0^{s+\frac12}(\Omega), \qquad s\in\left(0,\frac12\right],
  \end{equation*}
  is bounded, 
  i.e.,  there exists a constant $M$ such that  $\forall f \in H^{-s}(\Gamma)$, $\forall v \in  H^{s+\frac12}_0(\Omega)$:   \begin{equation}
    \label{eq:bound-on-M}
    {}_{H^{-s-\frac12}(\Omega)}\langle\LD f, v\rangle_{H^{s+\frac12}_0(\Omega)} := \int_\Gamma f v \d \Gamma \leq  
    M \| f \|_{-s, \Gamma} \| v \|_{s+\frac12,\Omega}.
  \end{equation}
\end{theo}

\begin{proof}
For all $\varphi\in\D(\Omega)$ and $f \in H^{-s}(\Gamma)$, $s \in \left(0,\frac12\right]$, we can write
\begin{equation*}
\left| \int_\Omega\int_\Gamma \delta(x-y)f(y)\varphi(x)\d\Gamma_y \d
\Omega_x \right|=
\left| \int_\Gamma f(y)\varphi(y) \d\Gamma_y \right| \leq \| f \|_{-s,\Gamma}
\| \varphi\|_{s,\Gamma}.
\end{equation*}
We apply Theorem~\ref{theo:trace} to the second argument on the right hand side
\begin{equation*}
\langle   \LD f,\varphi \rangle \leq  C_T\| f \|_{-s,\Gamma}
\| \varphi\|_{s+\frac12,\Omega},
\end{equation*}
and the first part of the thesis follows with a density argument. 
\end{proof}

We observe that for any $f\in H^{-s}(\Gamma)$, with $s \leq \frac12$,
the operator $\LD f$ belongs to $H^{-1}(\Omega)$, and can be used as a
standard source term in the variational formulation of the Poisson
problem. In particular, if $f \in H^m(\Gamma)$, with $m\geq0$, it is
also in $H^{-\epsilon}(\Gamma)$ for any $\epsilon>0$.

The final variational problem can be formulated as follows.

\begin{problem}[Variational]
\label{pb:variational}
Given $b \in L^2(\Omega)$ and $f\in H^{-s}(\Gamma)$,
$s\in\left(0,\frac12\right]$, find $u \in H^1_0(\Omega)$ such that:
\begin{equation}
  \label{eq:variational-u-1}
  \begin{aligned}
    &(\nabla u, \nabla v)  =   (b,v)+ \duality{\LD  f, v} \qquad&& \forall v \in H^1_0(\Omega),  
      \end{aligned}
\end{equation}
where
\begin{equation}
  \label{eq:variational-u-2}
  \begin{aligned}
    & \duality{ \LD   f, v} :=  \int_\Gamma  f ~v \d \Gamma  && \forall v \in H^1_0(\Omega).
  \end{aligned}
\end{equation}
\end{problem}

We indicate with $(\cdot, \cdot)$ the $L^2(\Omega)$ inner product, and
with $\duality{\cdot, \cdot}$ the duality product between
$H^1_0(\Omega)$ and $H^{-1}(\Omega)$.

In particular, we have that the exact solution $u$ satisfies the
following regularity result:
\begin{lemma}[Continuous dependence on data]
  \label{lem:continuous dependence on data}
  Problem~\ref{pb:variational} is well posed and has a unique solution
  in $H^{\frac32-s}(\Omega)$ that satisfies
  
\begin{equation}
  \label{eq:continuous-estimate}
  |u |_{\frac32-s,\Omega} \leq C \| b +   \LD f  \|_{-\frac12-s,\Omega} \leq C 
  (\|  b \|_{0,\Omega} +  \|  f \|_{-s, \Gamma}).
\end{equation}
\end{lemma}
\begin{proof}
  For convex domains, with Lipschitz boundary $\partial \Omega$,
  Problem~\ref{pb:variational} is 2-regular, and admits a unique
  solution that satisfies the estimate:
  \begin{equation}
    \label{eq:2-regularity}
    |u|_{k+2,\Omega} \leq \|b+\LD f\|_{k,\Omega},
  \end{equation}
  whenever the right hand side of Problem~\ref{pb:variational} is in
  $H^{k}(\Omega)$, $-1 \leq k\leq 0$.

  Exploiting Theorem~\ref{theo:regularity M} and taking
  $k = -\frac12-s$ in~\eqref{eq:2-regularity}, we obtain the thesis.
\end{proof}

This formulation does not require any approximation of the Dirac delta, since the regularity of $\LD$ is compatible with test functions in $H^1_0(\Omega)$, allowing a natural approximation by Galerkin methods, using, for example, finite elements~\cite{BoffiGastaldi-2003-a, BoffiGastaldi-2007-Numerical-stability-0, Heltai-2008-a}.

\section{Finite element approximation}
\label{sec:numer-appr}

We consider a decompositions of $\Omega$ into the triangulation $\Omega_{h}$, consisting of cells $K$ (quadrilaterals in 2D, and hexahedra in 3D) such that
\begin{enumerate}
\item $\overline{\Omega} = \cup \{ \overline{K} \in \Omega_{h} \}$;
\item
Any two cells $K,K'$ only intersect in common faces, edges, or vertices.
On $\Omega_{h}$ 
we define 
the finite dimensional
subspace ${W}^\ell_{h} \subset H^1_0(\Omega)$, 
such that
\begin{alignat}{5}
\label{eq: functional space v h}
{W}^\ell_h &:= \Bigl\{ {u}_h \in H^1_0(\Omega)\,&&\big|\, {u}_{h|K}  &&\in \mathcal{Q}^\ell(K), \, K &&\in \Omega_h \Bigr\} &&\equiv \text{span}\{ {v}_{h}^{i} \}_{i=1}^{N_{W}},
\end{alignat}
where $\mathcal{Q}^\ell(K)$ 
is a tensor product polynomial space of degree $\ell$ 
on the cells $K$, and $N_W$ 
is the dimension of the finite dimensional space.

\end{enumerate}

The index $h$ stands for the maximum radius of $K$, and we assume that $\Omega_h$ is shape regular, i.e., $\rho_K \leq h \leq C \rho_K$ where $\rho_K$ is the radius of the largest ball contained in $K$, and the inequality is valid for a generic constant $C>0$ independent on $h$. We assume, moreover, that 
\begin{equation}
  \label{eq:condition-on-K-Gamma}
  \begin{split}
    | K \cap \Gamma | &\leq C_0 h, \qquad \forall K \in \Omega_h,
  \end{split}
\end{equation}
where $C_0$ 
is a generic constant independent on $h$, and with $ | K \cap \Gamma |$ we indicate the $n-1$ dimensional Hausdorff measure of the portion of $\Gamma$ that lies in $K$. These assumptions are generally satisfied whenever $h$ is sufficiently small to capture all the geometrical features of both $\Gamma$ and $\Omega$. 

Using the space $W^\ell_h$, the finite dimensional version of Problem~\ref{pb:variational} can be written as
\begin{problem}[Discrete]
\label{pb:discrete}
  Given $b\in L^2(\Omega)$,
  $ f \in  H^{-s}(\Gamma)$, $s\in\left(0,\frac12\right]$, find $u_h \in W^\ell_h \subset H^1_0(\Omega)$ such that
\begin{equation}
\label{eq:discrete_problem-1}
\begin{aligned}
  & \duality{\nabla u_h, \nabla v_h} = (b, v_h) + \duality{
    \LD f, v_h} && \forall v_h \in W^\ell_h,
\end{aligned}
\end{equation}
where
\begin{equation}
\label{eq:discrete_problem-2}
\begin{aligned}
&  \text{ }  \duality{ \LD f ,v_h} := \phantom{-}\int_{\Gamma} f ~v_h \d
  \Gamma && \forall v_h \in W^\ell_h.
\end{aligned}
\end{equation}
\end{problem}

\begin{theo}[A-priori error estimates]
\label{theo:a-priori-standard}
  The finite element solution $u_h$ of Problem~\ref{pb:discrete} satisfies the following \emph{a-priori} error estimates for $f \in H^{-s}(\Gamma)$, $s \in \left(0,\frac12\right]$
  \begin{equation}
    \label{eq:general-estimate}
    \|u- u_h\|_{m,\Omega}\leq C h^{\frac32- s - m} (\|b\|_{0,\Omega} + \|f\|_{-s ,\Gamma}), \qquad m=0,1
  \end{equation}
  and
  \begin{equation}
    \label{eq:general-estimate-f-in-L2}
    \|u- u_h\|_{m,\Omega}\leq C h^{\frac32-\epsilon - m} (\|b\|_{0,\Omega} + \|f\|_{0 ,\Gamma}), \qquad \forall \epsilon \in \left(0,\frac12\right], \qquad m=0,1
  \end{equation}
  for $f \in L^2(\Gamma)$, where $u$ is the solution to Problem~\ref{pb:variational}.
  
\end{theo}
\begin{proof}
For the finite element approximation defined in Problem~\ref{pb:discrete}, we expect error estimates of the type
\begin{equation}
\label{eq:error estimates standard}
 \|u- u_h\|_{m,\Omega}\leq C h^{k-m} |u|_{k,\Omega}, \qquad m \leq k \leq \ell+1, \qquad m=0,1,
\end{equation}
where $u$ is the solution to Problem~\ref{pb:variational} and $u_h$ is
the solution to Problem~\ref{pb:discrete} ~\cite{ciarlet78}. The
thesis follows applying Lemma~\ref{lem:continuous dependence on data}
with $k=\frac32-s$.

\end{proof}

The low regularity of $\LD f$ affects the numerical approximation of the problem, and produces sub-optimal (w.r.t. to the approximation degree $\ell$) error estimates when standard finite elements are used. This phenomena is known and was observed in the literature of the variational formulation of the  Immersed Boundary Method~\cite{BoffiGastaldiHeltaiPeskin-2008-a, Heltai2012b}, motivating this work.
 
The solution is expected to be at least $H^2(\Omega\setminus\Gamma)$, and it seems reasonable to assume that  the numerical solution to such problems is sub-optimal only in a small neighbourhood of $\Gamma$. 

In this work we show that the numerical solution obtained by solving
Problem~\ref{pb:discrete} with the classical finite element method
applied directly to Problem~\ref{pb:variational}, results in
approximate solutions which converge optimally when we measure the
error with properly chosen weighted Sobolev norms. Similar ideas are
presented in~\cite{Agnelli2014}, and
in~\cite{Dangelo2012,DAngelo2008}. Rigorous proofs of the well
posedeness of the approximation of Poisson type problems where the
source is given by a singular measure and the domain is a convex
polygonal or polyhedral domain are given in~\cite{Drelichman2018}.

A generalization to non-convex domains and to Stokes problem is
available in~\cite{Otarola2018}, provided that singularities are away
from the boundary of the domain $\Omega$.

\section{Weighted Sobolev spaces} 
\label{sec:weight-sobol-spac}

The set $\Gamma$ has zero measure, so it is reasonable to introduce weighted Sobolev norms, where the weight is chosen to be an appropriate power of the distance from the co-dimension one surface $\Gamma$. Such weight mitigates the jump of the gradient of $u$ across $\Gamma$, allowing a more regular variational formulation of the problem in the entire $\Omega$ in a rigorous and numerically convenient way.

We define the Hilbert space of measurable functions (see, e.g.,~\cite{Kufner1985, Turesson2000})
\begin{equation}
L^2_{\alpha}(\Omega)=\left\{u(x):\Omega\to\RR \text{ s.t. }\left(\int_\Omega u(x)^2 d^{2\alpha}(x)\ddx \right)^{\frac12} < \infty, \,\alpha \in \left(-\tfrac12,\tfrac12\right)\right\} 
\label{eq:functional-space}
\end{equation}
equipped with the  scalar product 
\begin{equation}
(u,v)_{\alpha} := \int_\Omega u(x)v(x) d^{2\alpha}(x)\ddx. 
\label{eq:scalar-product}
\end{equation}
In  \eqref{eq:functional-space} and  \eqref{eq:scalar-product}, $d$ is the distance between the point $x$ and the surface $\Gamma$, that is 
\begin{equation}
\label{eq:distance}
d(x)=\text{dist}(x,\Gamma).
\end{equation}

For any $\alpha$ in $\left(-\tfrac 12,\tfrac 12\right)$, the weighting function $w:\Re^n\to \Re_+$ defined by
  $w(x) := \text{dist}(x, \Gamma)^{2\alpha}$,  is a Muckenhoupt class $A_2$-weight, that is
  \begin{equation}
    \label{eq:muckencazz}
    \sup_{B=B_r(x), x \in \Re^n, r>0} \left( \frac{1}{|B|} \int_B w(x) \d x \right) 
    \left( \frac{1}{|B|} \int_B w(x)^{-1} \d x \right) < +\infty,
  \end{equation}
  where $B_r (x)$ is the ball centered at  $x$ with radius $r$, and $|B|$ is its measure~\cite{Muckenhoupt1972, Kufner1985, Turesson2000}.

The identity $\duality{u,v} = \duality{u d^{\alpha}, d^{-\alpha}v}$
implies that $L^2_{-\alpha}(\Omega)$ is contained in the dual space of  $L^2_{\alpha}(\Omega)$. 
We denote by $\|{u}\|_{0,\alpha,\Omega}$ the norm of a function $u$ in $L^2_{\alpha}(\Omega)$, i.e., $\|d^{\alpha}u\|_{0,0, \Omega}$, where $d$ is defined in \eqref{eq:distance}. From now on we will use the notation $\|u\|_{0,0, \Omega}\equiv \|u\|_{0, \Omega}$ for all $u \in L^2(\Omega)$.
The duality pairing between $L^2_{\alpha}(\Omega)$ and $L^2_{-\alpha}(\Omega)$ is indicated with
$\langle u,v\rangle$. 
The usual inequality  $\langle u,v\rangle = \langle d^\alpha u,d^{-\alpha} v\rangle \leq \|u\|_{0,\alpha, \Omega}~\|v\|_{0,-\alpha, \Omega}$ follows from Schwartz inequality in $L^2(\Omega)$.

Similarly, we define the weighted Sobolev spaces
\[
H^s_\alpha(\Omega)=\{u \text{ such that } D^\gamma u\in L^2_\alpha(\Omega),|\gamma|\leq s\},
\] 
where $s\in \mathbb{N}$, $\gamma$ is a multi-index and $D^\gamma$  its corresponding distributional derivative. These  weighted Sobolev spaces can be equipped with the following seminorms and norms
\[
|u|_{s,\alpha, \Omega}:=\left(\sum_{|\gamma|=s}\|D^\gamma u\|^2_{0, \alpha,\Omega}\right)^{\frac12}, \qquad  \|u\|_{s,\alpha, \Omega}:=\left(\sum_{k=0}^s|u|^2_{k,\alpha, \Omega}\right)^{\frac12},
\]
and we define the Kondratiev type weighted spaces $V^s_\alpha(\Omega)$ using the same seminorms, but weighting them differently according to the index of derivation, i.e., 
\[
V^s_\alpha := \{ u \text{ such that } D^\gamma u\in L^2_{\alpha-j}(\Omega),|\gamma|= s-j, \quad j =0, \dots, s \},
\]
equipped with the following norm:
\[
||| u |||_{s, \alpha, \Omega} := \left( \sum_{j = 0}^s |u|_{j, \alpha-s+j,\Omega}\right)^{\frac12}.
\]
For example, $|||u|||_{1,\alpha,\Omega} := |u|_{1,\alpha,\Omega} + ||u||_{0,\alpha-1,\Omega}$.
The norms in $V^1_\alpha$ and $H^1_\alpha$ are equivalent, but not
uniformly with respect to $\alpha$ \cite{Kufner1985}.

We define the space 
\begin{equation}
\label{eq:definition-W}
W_{\alpha} := \{ u \in H^1_{\alpha}(\Omega) \text{ such that } u|_{\partial \Omega} = 0 \},
\end{equation}
with norm $\| \cdot \|_{1,\alpha,\Omega}$ and we denote with
$W'_{\alpha}$ its dual space with norm
\begin{equation}
  \label{eq:W-minus-alpha-norm}
  \| f \|_{-1,\alpha, \Omega} := \sup_{0 \neq v \in W_{\alpha}}
  \frac{\duality{f, v}}{\| v\|_{1, \alpha, \Omega}}.
\end{equation}

Notice that an equivalent definition is obtained using the
$|||\cdot|||$ norms.

\begin{lemma}
  \label{lem:embedding}
  Given $\alpha \in (-\frac12, \frac12)$ and $\epsilon\geq0$ such that
  $\epsilon+\alpha\in (-\frac12,\frac12)$, the embeddings $H^m_{\alpha}
  \hookrightarrow H^m_{\alpha+\epsilon}$ and $W'_{\alpha+\epsilon}
  \hookrightarrow W'_{\alpha}$  are continuous. 
\end{lemma}
\begin{proof}
  For $\epsilon \ge 0$ the function $d^{2\epsilon}$ is bounded and
  continuous on $\overline \Omega$, by H\"older inequality, we have
  that
 \begin{equation}
    \| u \|^2_{0,\alpha+\epsilon, \Omega} = \| u^2 d^{2\alpha}
    d^{2\epsilon}\|_{L^1(\Omega)} \leq  \| u^2 d^{2\alpha}\|_{L^1(\Omega)} \| d^{2\epsilon}\|_{L^\infty(\Omega)} = 
    \|u\|^2_{0,\alpha, \Omega} \| d^{2\epsilon}\|_{L^\infty(\Omega)}
  \label{eq:embedding}
\end{equation}

The thesis follows by applying the inequality in
Equation~\eqref{eq:embedding} to $D^{\gamma}u$ for all
$|\gamma| \leq m$, and we get
\begin{equation}
  \label{eq:continuous-embedding}
  \|u\|_{m,\alpha+\epsilon,\Omega} \leq C_{\epsilon} \|u\|_{m,\alpha,\Omega},
\end{equation}
where $C_{\epsilon} = \max_{x\in\Omega} d^\epsilon(x)$.

The dual case follows applying the definition of the dual norm.
\end{proof}

\begin{theo}[Isomorphism of $-\Delta$]
  \label{theo:isomorphism}
  The Laplace operator $-\Delta$ is an isomorphism from $W_{\alpha}$ to
  $W'_{-\alpha}$, for any $\alpha$ in $(-\frac 12,\frac 12)$.
\end{theo}
\begin{proof}
  The proof is an immediate consequence of~\cite[Theroem
  2.7]{Drelichman2018}: taking $k=n-1$ and $p=2$ in that
  theorem, one obtains that the Laplace operator $-\Delta$ is an
  isomorphism from $W_{\alpha}$ to $(W_{-\alpha})'$ (and we will write
  this space as $W'_{-\alpha}$).
\end{proof}

\begin{lemma}[Weighted space of Dirac terms]
  \label{lem:weighted space of dirac terms}
  The operator $  \LD f$ defined in Problem~\eqref{eq:problem_dirac} is in $W'_{-\alpha}$ for any $\alpha$ in $[0,\frac12)$.
\end{lemma}

\begin{proof}
  Since both $z$ and $p$ belong to $H^1(\Omega)$, by
  Lemma~\ref{lem:embedding}, they also belongs to
  $H^1_\alpha(\Omega)$, for $\alpha \in
  [0,1/2)$. Theorem~\ref{theo:isomorphism} applied to $z+p$ implies
  that the function $-\Delta (z+p) := b + \LD f$ belongs to
  $W'_{-\alpha}$, independently on the choice of $b\in L^2(\Omega)$
  and $f \in H^{-s}(\Gamma)$, with $s\in \left(0,\frac12\right]$.
\end{proof}

The weighted variational problem can be written as follows

\begin{problem}[Weighted variational]
  \label{pb:weighted}
  Given $b\in L^2(\Omega)$ and $ f\in H^{-s}(\Gamma)$, for
  any $\alpha$ in $[0,\frac12)$, find $u \in W_{\alpha}$ such that
  \begin{equation}
\label{eq:problem_weighted-1}
\begin{aligned}
& \duality{\nabla u, \nabla v} =  \duality{b, v}  + \duality{  \LD   f, v} && \forall v \text{ in } W_{-\alpha}, 
\end{aligned}
\end{equation}
where
  \begin{equation}
\label{eq:problem_weighted-2}
\begin{aligned}
& \duality{  \LD f,v}  := \phantom{-}\int_\Gamma   f~v \d \Gamma  && \forall v \in W_{-\alpha}.
\end{aligned}
\end{equation}
\end{problem}

\begin{lemma}[Continuous dependence on data -- weighted case]
  \label{lem:continuous dependence on data - weighted case}
  Problem~\ref{pb:weighted} is well posed and has a unique solution
  that satisfies
  
\begin{equation}
  \label{eq:continuous-estimate}
  \|u \|_{1,\alpha,\Omega} \leq C \| b +   \LD f  \|_{-1, -\alpha,\Omega} \leq C 
    (\|  b \|_{0,0,\Omega} +  \|  f \|_{-s, \Gamma}).
\end{equation}
\end{lemma}
\begin{proof}
  The chain of inequalities come directly from
  Theorem~\ref{theo:isomorphism}, Lemma~\ref{lem:embedding}, and from Theorem~\ref{theo:regularity M}. For the second inequality we first
  consider the term
  \begin{equation*}
    \| b  \|_{-1, -\alpha,\Omega} :=  \sup_{0 \neq v \in W_{-\alpha}}
  \frac{\duality{b, v}}{\| v\|_{1, -\alpha, \Omega}}.
\end{equation*}

In particular,
\begin{equation*}
  \duality{b, v} \leq \| b \|_{0,\alpha,\Omega}\| v
  \|_{0,-\alpha,\Omega} \leq \sqrt{\| d^{2\alpha}
    \|_{L^\infty(\Omega)}}  \| b\|_{0,0,\Omega}  \|v\|_{1,-\alpha,\Omega},
\end{equation*}
since
\begin{equation*}
  \| b \|^2_{0,\alpha,\Omega} = \| b^2 d^{2\alpha} \|_{L^1(\Omega)}
  \leq \| b\|_{0,0,\Omega} \sqrt{\| d^{2\alpha} \|_{L^\infty(\Omega)}}.
\end{equation*}

For the second term, we exploit the definition of the operator $\LD$:
\begin{equation*}
\duality{\LD f, v} := \int_{\Gamma} f v \d\Gamma \leq \| f
\|_{-s,\Gamma} \| v \|_{s,\Gamma} \leq \| f
\|_{-s,\Gamma} \| v \|_{\frac12,\Gamma} ,
\end{equation*}
and use a trace inequality and Lemma~\ref{lem:embedding}:
\begin{equation*}
  \| v \|_{\frac12,\Gamma} \leq C_T \|v\|_{1,0,\Omega} \leq C_T
  C_\alpha \|v\|_{1,-\alpha,\Omega}.
\end{equation*}

The thesis follows with $C = \max\left\{\sqrt{\| d^{2\alpha}
  \|_{L^\infty(\Omega)}}, C_T C_\alpha\right\}$.
\end{proof}

\begin{remark}[Non-weighted formulation]
  We remark here that the standard variational formulation, presented
  in Problem~\ref{pb:variational}, is a special case of
  Problem~\ref{pb:weighted} when the power of the distance is taken to
  be zero. In this case $W_\alpha = W_{-\alpha} = H^1_0(\Omega)$ and
  we recover Problem~\ref{pb:variational}. Moreover, the embedding
  $W_{-\alpha} \hookrightarrow H^1_0(\Omega)$ implies that the
  (unique) solution to Problem~\ref{pb:variational} is the same as the
  solution to Problem~\ref{pb:weighted}, and it belongs to
  $W_{\alpha} \cap H^{\frac32-s}(\Omega)$ for any
  $\alpha \in \left[0,\frac12\right)$.
\end{remark}


\begin{lemma}[Finite dimensional subspace of $W_\alpha$]
\label{lem:Wh-in-Walpha}
The finite dimensional space ${W}^\ell_{h} $ is a subspace of  $W_{\alpha}$ for all $\alpha$ in $ (-\frac12,\frac12)$. 
\end{lemma}
\begin{proof}
  The space ${W}^\ell_{h}$ is a subspace of
  $H^1_0(\Omega) \cap C^0(\overline\Omega) \cap W^{1,\infty}(\Omega)$,
  i.e., finite element functions are Lipschitz continuous. Following
  the same lines of Lemma~\ref{lem:embedding}, we have:
  \begin{equation}
    \label{eq:wh-embedding}
    \| u_h \|^2_{0,\alpha,\Omega} = \| u_h^2 d^{2\alpha}
    \|_{L^1(\Omega)} \leq \| u_h^2 \|_{L^{\infty}(\Omega)}
    \|d^{2\alpha}\|_{L^1(\Omega)} \quad \forall u_h \in W^\ell_{h},
    \quad \alpha
    \in \left(-\frac12, \frac12\right),
  \end{equation}
  and similarly for the gradients:
   \begin{equation}
    \label{eq:wh-embedding-gradients}
    \| \nabla u_h \|^2_{0,\alpha,\Omega} \leq \| (\nabla u_h)^2 \|_{L^{\infty}(\Omega)}
    \|d^{2\alpha}\|_{L^1(\Omega)} \quad \forall u_h \in W^\ell_{h}\quad \alpha
    \in \left(-\frac12, \frac12\right).
  \end{equation}
\end{proof}

This allows us to use the same finite dimensional space for both test
and trial functions in the numerical implementation, which is
identical to the classical Galerkin approximation.

\begin{theo}[Stability of Problem~\ref{pb:weighted}]
  \label{theo:stability}
Let $|\alpha| < \frac12$, then there exist $C_1$, $C_2$, and $C_3$ such that
\[
\sup_{v_h\in W_h^\ell}\frac{\duality{\nabla u_h, \nabla v_h}}{\|
  \nabla v_h \|_{0,-\alpha,\Omega}} \ge C_1 \| \nabla u_h \|_{0,\alpha,\Omega}, \qquad 
\sup_{v_h\in W_h^\ell}\frac{\duality{\nabla u_h, \nabla v_h}}{\|
  \nabla u_h \|_{0,\alpha,\Omega}} \ge C_2 \| \nabla v_h \|_{0,-\alpha,\Omega}.
\]
The Galerkin approximation (Problem~\ref{pb:discrete}) of Problem~\ref{pb:weighted} in
$W^\ell_{h}$ is stable, and the following error estimate holds
\begin{equation}
\label{eq:optimal_error_estimate}
\| u- u_h \|_{1,\alpha,\Omega} \lesssim C_3 \inf_{v_h\in W_h^\ell}\| u-v_h \|_{1,\alpha,\Omega}.
\end{equation}
\end{theo} 
\begin{proof}
  The two inf-sup conditions follow from~\cite[Theorem
  3.4]{Drelichman2018} setting $k=n-1$, and $p=2$, and observing that
  $2\alpha = \lambda$. Estimate~\ref{eq:optimal_error_estimate}
  follows using standard results for Petrov-Galerkin approximations.
\end{proof}

From the point of view of the implementation,
Problem~\ref{pb:discrete} is a standard finite element problem. The
only difficulty is given by the integration of the test functions on
the surface $\Gamma$, which is not aligned with the grid where $v_h$
are defined. This is usually done by some quadrature formulas, where
the integration is performed approximately using a fixed number of
points on $\Gamma$ (see, e.g.,~\cite{Heltai2012b}). It is possible to
choose a quadrature formula such that the error induced by the
numerical approximation of the integral over $\Gamma$ is of higher
order with respect to the overall order of accuracy, by making sure
that the integration on $\Gamma$ is performed by splitting the curve
or surface at the boundaries of the elements of the triangulation used
for $\Omega$, and by using enough quadrature points.

%

\subsection{Interpolation estimates in weighted Sobolev spaces}%

The estimate provided in Theorem~\ref{theo:stability} does not exploit
the particular structure of our singular forcing term. We adapt the
construction of the interpolation operator from $W_\alpha$ to
$W_h^\ell$ presented in~\cite{Dangelo2012} to the co-dimensional one
case, and show that optimal error convergence rates can be achieved
when the singular forcing term has the form introduced in this paper.

For each element $K \in \Omega_h$ we define the quantities
\[
{d}_K:=\text{dist}(K,\Gamma), \qquad \overline{d}_K:=\max_{x\in K} \text{dist}(K,\Gamma), \qquad h_K:=\text{diam}(K).
\]

We assume that the mesh is quasi-uniform, i.e., 
 there exist two positive constants $c$ and $C$ such that
\begin{equation}
\label{eq:scaling_of_hK}
\begin{split}
&c h\leq h_K\leq C h.
\end{split}
\end{equation}
We split the mesh in two parts, one containing all 
elements close to $\Gamma$, that is 
\[
\Omega_h^{\text{in}}:= \{K\in \Omega_h \text{ such that } \bar d_K\leq \sigma h
 \},
\]
where $\sigma$ is a fixed safety coefficient, and the other, $\Omega_h^{\text{out}} := \Omega_h \setminus \Omega_h^{\text{in}} $.

\begin{figure}\centering
  \setlength{\unitlength}{0.7\textwidth}
  \begin{picture}(1,.5)
    \put(0,0){\includegraphics[width=0.7\textwidth]{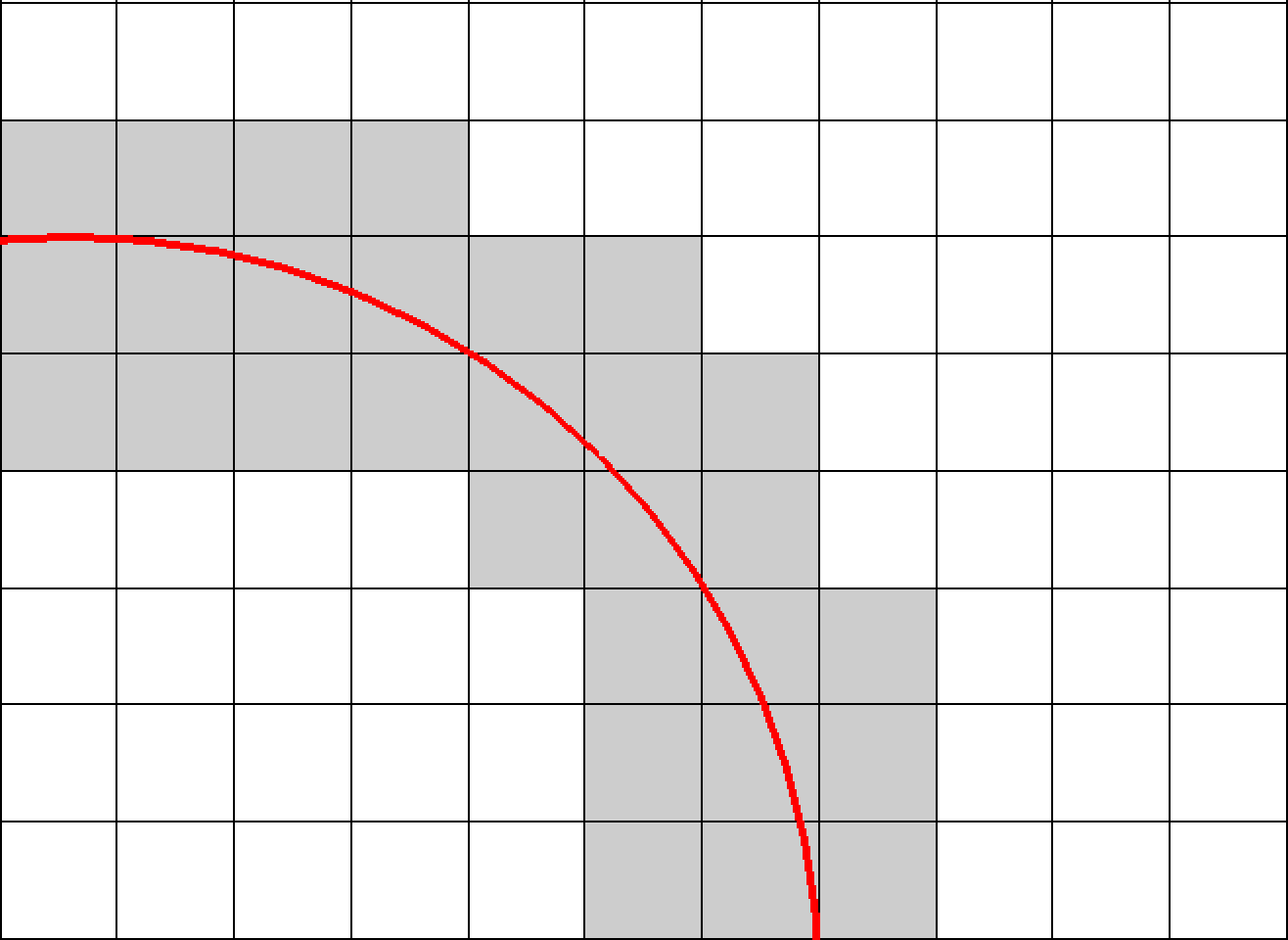}}
    \put(.285,.3){$\Omega_h^{\text{out}}$}
    \put(.745,.48){$\Omega_h^{\text{out}}$}
    \put(.47,.21){$\Omega_h^{\text{in}}$}
    \put(.3,.5){$\Gamma$}
  \end{picture} 
  \caption{The splitting of the computational mesh in $\Omega_h^{\text{in}}$ (grey elements in the picture) and $\Omega_h^{\text{out}}$ (white elements in the picture).}
\end{figure}
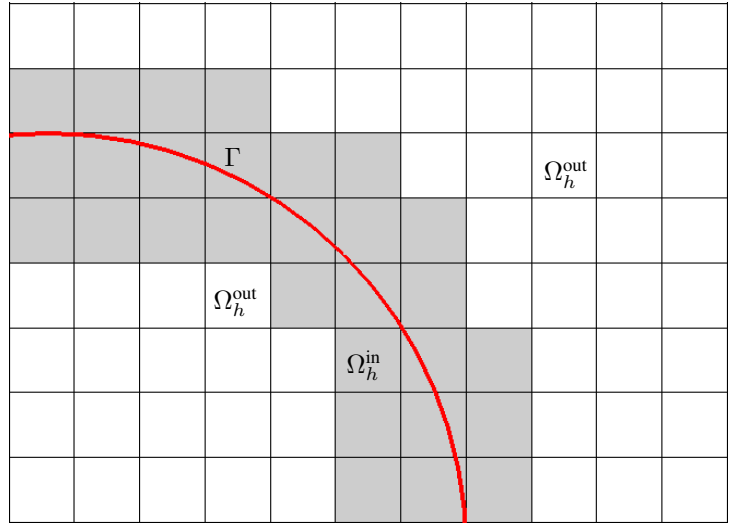

It can be shown that
\begin{equation}
\label{eq:mesh-properties}
  \begin{split}
  &d_K\lesssim h_K, \quad
  h_K\lesssim \overline{d}_K \lesssim h_K, \quad \forall K \in \Omega_h^{\text{in}},
  \\
&\overline{d}_K \lesssim d_K, \qquad \quad \forall K \in 
 \Omega_h^{\text{out}},
 \end{split}
  \end{equation}
where the notation $a\lesssim b$ is used to indicate that there exists a constant $C>0$ such that $a \leq C b$. 

We define the folloing discrete norm
\begin{equation}
\label{eq:discrete-norm}
\| u_h \|^2_{h,\alpha}:=\sum_{K\in \Omega_h}(\overline{d}_K)^{2\alpha}\| u_h \|^2_{0,K}.
\end{equation}

\begin{lemma}
Let $|\alpha|< t$, $t\in [0,\frac12)$, then 
 the following norms are equivalent
\begin{equation}
\label{eq:equivalent-norms}
\| u_h \|_{h,\alpha}\lesssim \| u_h \|_{0,\alpha,\Omega}
\lesssim \| u_h \|_{h,\alpha}, \qquad \forall u_h \in W_h^r
\end{equation}
 where the constants of the inequalities depend only on $t$.
\end{lemma}
\begin{proof}
The proof follows very closely~\cite[Lemma 3.2]{Dangelo2012}. We
consider $\alpha \geq 0$ (the other case follows similarly).
Let $K\in \Omega_h$ and $x\in K$; we have $d(x)^{2\alpha} \leq
(\overline{d}_K)^{2\alpha}$,  so that the first part of the inequality
$\|u_h\|_{0,\alpha,\Omega}  \leq
\|u_h\|_{h,\alpha}$ follows trivially.

The second part of the inequality follows by distinguishing two cases: If $K
\in \Ohout$, we use Equation~\eqref{eq:mesh-properties}, and we have
$\overline{d}_K^{2\alpha} \| u_h \|^2_{0,0,K} \lesssim d^{2\alpha}_K
\| u_h \|^2_{0,0,K} \leq \| d^{\alpha} u_h \|^2_{0,0,K}$. 

Following~\cite[Lemma 3.2]{Dangelo2012}, we show that a
similar estimate holds true if $K \in \Ohin$. Let $\hat
K$ be the reference element and let $F_K : \hat K \to K$ be the affine
transformation, mapping $\hat K$  onto the actual element $K$. Let
$\hat u_h = u_h  \circ F_K$, and let $\hat \Gamma_K = F_K^{-1}(\Gamma)$, such that
$\Gamma$ is the $K$ image of $\hat \Gamma_K$ under $F_K$, and let $\hat
d(\hat x) = \text{dist}(\hat x, \hat \Gamma_K)$.

Thanks to shape regularity, the eigenvalues of the Jacobian
matrix of $F_K$ are uniformly upper and lower bounded by $h_K$. Hence,
distances are transformed according to $d(F_K(\hat x)) \gtrsim h_K
\hat d(\hat x)$. As a result, 
\begin{equation*}
  \|d^\alpha u_h \|^{2}_{0,0,K} = \int_K d^{2\alpha}u_h^2 =
  \frac{|K|}{|\hat K|} \int_{\hat K} [ d(F_K(\hat x))]^{2\alpha} \hat
  u_h^2 \gtrsim h_K^{2\alpha} \frac{|K|}{|\hat K|}  \int_{\hat K}  \hat d^{2\alpha} \hat
  u_h^2.
\end{equation*}

\begin{figure}
  \centering
  \includegraphics[]{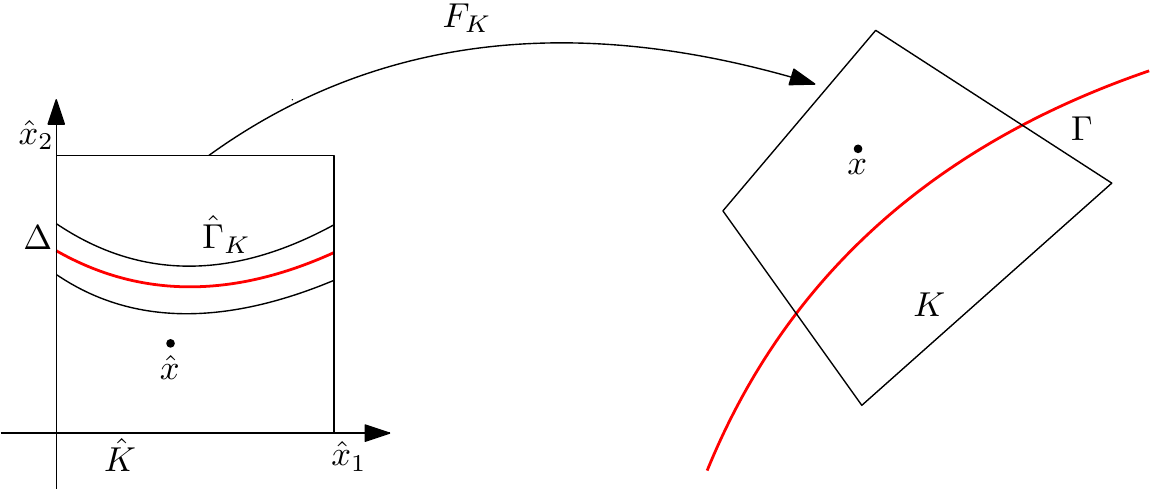}
  \caption{Reference element $\hat K$, its transformation to the real
    element $K$ and visualisation of $\Gamma$, $\hat \Gamma$, and $\Delta$.}
  \label{fig:delta-definition}
\end{figure}

Let us introduce the subset $\hat K_\Delta = \{\hat x \in \hat K :
\text{dist}(\hat x, \hat \Gamma_K) > \Delta \}$, where $\Delta > 0$ is a
parameter (see Figure~\ref{fig:delta-definition}); we have 
\begin{equation*}
  \Delta^{2\alpha} \| \hat u_h \|^2_{0,0,\hat K_\Delta} \leq
  \int_{\hat K} \hat d^{2\alpha} \hat u_h^2.
\end{equation*}
Note that for $\Delta$ small, $\inf_K |\hat K_\Delta |$ cannot
degenerate with respect to $|\hat K | \gtrsim 1$.

We can estimate $|\hat K|-|\hat K_\Delta | \lesssim \Delta$,
irrespective of the position of $\hat \Gamma_K$; hence, choosing
$\Delta$ sufficiently small, we have
$| \hat K_\Delta| \geq c' \Delta |\hat K |$, where the constant
$c' = 1 -O(\Delta)$ depends on $\Delta$ but not on $K\in
\Ohin$. 

Notice that, when the surface is of co-dimension one, as in our case,
the constant $c'$ is of order $O(\Delta)$ instead of order
$)(\Delta^2)$ as in the case presented by~\cite{Dangelo2012}.

Similarly, it can be seen that
$\| \hat u_h\|^2_{0,0,\hat K_\Delta} \geq c_\Delta \| \hat{u}_h
\|^2_{0,0,\hat{K}} \quad \forall \hat{u}_h \in \mathcal{Q}^\ell(\hat K)$ (it
suffices to see that the estimates hold for the local basis functions
on $\hat K$), where again the constant $c_\Delta$ depends only on
$\Delta$ and not on the shape of $\hat K_\Delta$. Hence, using $\alpha
\leq t $ and Equation~\eqref{eq:mesh-properties}, we conclude
\begin{multline*}
  \| d^\alpha u_h\|^2_{0,0,K} \gtrsim c_\Delta \Delta^{2\alpha}
  h_K^{2\alpha}  \frac{|K|}{|\hat K|}  \|\hat u_h\|^2_{0,0,\hat K} \\
  \geq c_\Delta \Delta^{2 t}
  h_K^{2\alpha}  \frac{|K|}{|\hat K|}  \|\hat u_h\|^2_{0,0,\hat K}
 \gtrsim c_\Delta \Delta^{2 t}
  (\overline d_K)^{2\alpha} \|u_h\|^2_{0,0,K}.
\end{multline*}

\end{proof}

\begin{lemma} Given $u\in H^s(\Omega)$,  with  $0 \leq m \leq s\leq \ell$,
  $\alpha \in \left[0,\frac12\right)$, we have that
  \label{lem:weighted inequality}
  \begin{equation}
    \label{eq:weighted inequality}
    |u|_{m, \alpha, K} \lesssim \overline{d}_K ^{\ \alpha} h_K^{(s-m)} |u|_{s,0,K}, \qquad \forall K \in \Omega_h.
  \end{equation}
\end{lemma}
\begin{proof}
  By the definition of the weighted norm, we have that
  \begin{equation*}
    \| u \|_{0,\alpha,K}  = \| u d^\alpha \|_{0,0,K} \lesssim  \bar d_K^{\ \alpha}\| u \|_{0,0,K},
  \end{equation*}
  and similarly for the higher order semi-norms.

  A standard scaling argument, on the other hand, implies
  \begin{equation*}
    | u |_{m,0, K} \lesssim  h_K^{s-m} | u |_{s,0,K},
  \end{equation*}
  and the thesis follows by a combination of the two inequalities.
\end{proof}

\subsection{Convergence rate of the finite element approximation}
\label{sec:convergence rate}

We consider the nodal points $x_j$ of the basis functions $\phi_i$ that span the space $W^\ell_h$, i.e., $W^\ell_h := \text{span}\{\phi_i\}_{i=1}^{N}$, and $\phi_i(x_j) = \delta_{ij}$, where $\delta$ is the Kronecher delta.
We define the interpolation operator 
\begin{equation}
\Pi_h: H^{k+1}(\Omega\setminus\Gamma)\cap V^\ell_\epsilon(\Omega) \to W^\ell_h,
\label{eq:domain of Pi_h}
\end{equation}
as the operator that coincides with the standard finite element interpolation operator in $\Omega_h^{\text{out}}$ (see, for example, \cite{ciarlet78}), and that sets the degrees of freedom whose support belongs to $\Omega^{\text{in}}_h$ to zero, i.e.:
\begin{equation}
  \label{eq:definition pi h}
  \Pi_h u := \sum_{i \text{ s.t. } x_i \in \Ohout} u(x_i) \phi_i.
\end{equation}

Different interpolation operators could be defined for more general weighted Sobolev spaces, as in~\cite{Nochetto2016}. In this work we provide a generalisation of a result in \cite[section 3.3]{Dangelo2012}, for co-dimension one surfaces.
\begin{theo}[Properties of $\Pi_h$]
  \label{theo:properties of pi h}
  Let $\ell \leq k+1$ be such that the embedding $H^\ell(\Omega) \hookrightarrow L^\infty(\Omega)$ is continuous. For $\alpha \geq \epsilon$, $m \leq \ell$, the operator $\Pi_h$
  satisfies the following inequalities:
  \begin{equation}
    \label{eq:property of Pi_h}
    \begin{aligned}
      | u - \Pi_h u|_{m, 0, K} & \lesssim h_K^{k+1-m} |u|_{k+1,0,K}  \qquad && \text{ if } K \in \Ohout, \\
      | \Pi_h u |_{m, \alpha, K} & \lesssim h_k^{\ell-m +
        \alpha-\epsilon} ||| u |||_{\ell,\epsilon,\Delta_K} \qquad &&
      \text{ if } K \in \Ohin,
    \end{aligned}
  \end{equation}
where $\Delta_K$ is the set of all elements $K'$ in $\Ohout$ that share at least a node with $K$, i.e., 
\begin{equation*}
  \label{eq:delta K}
  \Delta_K := \{ K' \in \Ohout : \overline{K'}\cap\overline{K} \neq \emptyset \}.
\end{equation*}
\end{theo}
\begin{proof}
If $K$ belongs to $\Ohout$,  the first inequality follows from standard finite element theory (see, for example, \cite{ciarlet78}). Let's consider then $K$ in $\Ohin$. If $K$ does not share at least one node with $\Ohout$, $\Pi_h u$ is identically zero on $K$, and the second inequality follows trivially. Let us consider then the case in which $K$ shares the node $x_i$ with $K' \subset \Ohout$, and assume that for each element $K'$ there exists an affine transformation such that $K' = F_{K'}(\hat K)$, and $\hat u := u \circ F_{K'}^{-1}$.

In this case we can write 
\begin{equation}
\label{eq:proof part 1}
  | \Pi_h u |_{m,\alpha, K} \leq \sum_{K' \text{ s.t. } x_i \in \overline{K'}} {|u(x_i)|} {|\phi_i|_{m,\alpha, K'}}
\end{equation}
%
We start by estimating  $|u(x_i)|$,
\begin{equation}
  \begin{split}
  & \|u\|_{\infty, 0, K'} = \|\hat u\|_{\infty, 0, \hat K} \lesssim \| \hat u \|_{\ell, 0, \hat K} \\
   & \lesssim \left(\sum_{j=0}^\ell h_{K'}^{2j -n} |u|^2_{j, 0, K'}\right)^{\frac12} \\
   & = \left(\sum_{j=0}^{\bar \ell-1} h_{K'}^{2j -n} |u|^2_{j, 0, K'} + \sum_{j=\bar \ell}^\ell h_{K'}^{2j -n} |u|^2_{j, 0, K'}\right)^{\frac12} \qquad \bar \ell \text{ s.t. } \bar \ell + \epsilon - \ell > 0\\
   & = \left( \sum_{j=0}^{\bar \ell-1}h_{K'}^{2j -n}  \bar d_{K'}^{\ \ -2(j+\epsilon-\ell)}  |u|^2_{j, j+\epsilon-\ell, K'} + \sum_{j=\bar \ell}^\ell h_{K'}^{2j -n} d_{K'}^{\ \ -2(j+\epsilon-\ell)} |u|^2_{j, j+\epsilon-\ell, K'}\right)^{\frac12} \\
 & \lesssim h_{K'}^{\ \ \ell - \epsilon -\frac{n}2} \left( \sum_{j=0}^{\ell} |u|^2_{j, j+\epsilon-\ell, K'} \right)^{\frac12} = h_{K'}^{\ \ \ell - \epsilon -\frac{n}2} ||| u |||_{\ell, \epsilon, K'},
\end{split}
\end{equation}
where we used i) standard scaling arguments for Sobolev norms, ii) the fact that since $K'$ is in $\Delta_K$, and therefore it is sufficiently close to $\Gamma$, we can write $\bar d_{K'} \lesssim h_{K'}$ and $h_{K'} \lesssim d_{K'}$, iii)  and the fact that for a negative power $q^-$, we have $|u|_{m, 0, K'} \leq \bar d^{\ \ -q^-}_{K'} |u|_{m, q^-, K'}$ while for a positive power $q^+$ we have $|u|_{m, 0, K'} \leq d^{\ \ -q^+}_{K'} |u|_{m, q^+, K'}$.

To estimate the second term in the right-hand side of equation~\eqref{eq:proof part 1}, ${|\phi_i|_{m,\alpha, K'}}$, we use again a scaling argument for Sobolev norms:
\begin{equation}
  \label{eq:proof part II}
  \begin{split}
    |\phi_i|_{m,\alpha, K'} & \lesssim h_{K'}^{\ \ -m+\frac{n}2} |\hat \phi_i|_{m,\alpha, \hat K} \\
    & \lesssim  h_{K'}^{\ \ \alpha -m+\frac{n}2} 
    {|\hat \phi_i|_{m, 0, \hat K}} \\
    & \lesssim  h_{K'}^{\ \ \alpha -m+\frac{n}2},
  \end{split}
\end{equation}
where to obtain the last inequality we have used that   ${|\hat \phi_i|_{m, 0, \hat K}} {\lesssim 1}$.

Combining \eqref{eq:proof part 1} and~\eqref{eq:proof part II}, and summing over all $K' \in \Delta_K$ we get the second inequality of the thesis.

\end{proof}

We are now in the position to prove our main result. 
\begin{theo} 
\label{theo:main-estimate}
Let $u$ be the exact solution to Problem~\ref{pb:weighted}, and let
$u_h$ be the solution to Problem~\ref{pb:discrete} in $W_h^\ell$ with
$\ell \geq 1$. In particular
$u \in H^2(\Omega\setminus\Gamma) \cap H^{\frac32-s}(\Omega) \cap
W_{\alpha}$, and
\begin{equation}
|u-u_h|_{m,\alpha,\Omega}\lesssim h^{\frac32-s+\alpha-m} |||u|||_{\frac32-s,
  0, \Omega}
\end{equation}
where $s \in \left(0,\frac12\right]$, and $\alpha \in \left[0,\frac12\right)$.
\end{theo}

\begin{proof}
Consider $K\in \Ohout$. Using Lemma~\ref{lem:weighted inequality} and property~\eqref{eq:property of Pi_h}, we easily obtain
\begin{equation}
\begin{split}
|u-\Pi_h u|_{m,\alpha,K} & \lesssim \overline{d}_K^\alpha |u-\Pi_h u|_{m,0,K} \\
& \lesssim \overline{d}_K^\alpha h_K^{2-m} |u|_{2,0,K} \\
& \lesssim  h_K^{2-m} |u|_{2,0,K}, 
\end{split}
\end{equation}
where we used that $\overline{d}_K^\alpha \leq |\Omega|^\alpha = C$ for $\alpha \geq 0$.

Similarly, for $K$ in $\Ohin$, we have
\begin{equation}
\label{eq:error-estimate-on-Omega-in}
|u-\Pi_h u|_{m,\alpha,K}  \leq { |u|_{m,\alpha,K} }+ {|\Pi_h u|_{m,\alpha,K}}
\end{equation}
where we can estimate the first term in the right hand side of equation~\eqref{eq:error-estimate-on-Omega-in} using lemma~\ref{lem:weighted inequality},
and the second term using the second property in
equation~\eqref{eq:property of Pi_h}. Theorem~\ref{theo:stability}
implies the thesis.
\end{proof}


\section{Numerical validation}
\label{sec:validation}

All numerical examples provided in this section were obtained using an
open source code based on the \texttt{deal.II}
library~\cite{Bangerth2007,Bangerth2016b,MaierBardelloniHeltai-2016-a,AlzettaArndtBangerth-2018-a}
and on the \texttt{deal2lkit} toolkit~\cite{Sartori2018}. 

We construct an artificial problem with a known exact solution, and check the error estimates presented in the previous section. We begin with a simple two-dimensional problem, where we impose the Dirichlet data in order to produce a harmonic solution in the domain $\Omega\setminus\Gamma$, with a jump in the normal gradient along a circular interface, and extend the same test case to the three-dimensional setting.

Recalling our main result (Theorem~\ref{theo:main-estimate}), we have that the exact solution of this problem is in $H^2(\Omega\setminus\Gamma)$ (it is in fact analytic everywhere except across $\Gamma$) and globally $H^{\frac32}(\Omega)$, therefore we expect a convergence rate in weighted Sobolev spaces of the type:

\begin{equation}
  \label{eq:expected-rate}
  |u- u_h|_{m,\alpha,\Omega}\lesssim h^{\frac32-s-m+\alpha} |||u|||_{\frac32, 0, \Omega}, \qquad \alpha \in \left[0,\tfrac12\right), \qquad m=0,1.
\end{equation}

Notice that when $\alpha \to 1/2$, then the estimate tends to the
optimal case, while when $\alpha=0$, the estimate is classical
(suboptimal due to the lack of global regularity in the
solution). These situations occur very often in numerical simulations
of boundary value problems with interfaces, and the results we present
here show that a proper variational formulation of the interface terms
results in suboptimal convergence property of the finite element
scheme. However such sub-optimality is a local property due to the
non-matching nature of the discretisation, and it only influences the
solution  \emph{close to the surface $\Gamma$}. If we take this into
account when measuring the error, for example using a weighted Sobolev
norm as we do in this work, we recover the optimal estimate.

The results we presented in Theorem~\ref{theo:main-estimate} can be
applied also in the case of higher order approximations. However, the
theory shows that there would be no improvement in the \emph{global}
convergence rate beyond $\frac32-s-m+\alpha$, making the use of
polynomial order greater than one redundant. This is not to say that
the method would not benefit from higher order polynomiasl \emph{away}
from the co-dimension one surface $\Gamma$, even though the global
approximation error would not converge with the expected higher order
rate.

\subsection{Two-dimensional case}
\label{sec:two-dimensional-case}

In the two-dimensional case, we define the exact solution to be
\begin{equation}
  \begin{aligned}
    c = (0.3,0.3) \\ 
    r := x-c \\
  \end{aligned}
  \qquad
  u = 
  \begin{cases}
    -\ln(|r|) & \text{ if } |r| > 0.2, \\
    -\ln(0.2) & \text{ if } |r| \leq 0.2.
  \end{cases}
\label{eq:exact solution 2d}
\end{equation}

\begin{figure}
  \label{fig:exact 2d}
  \centering
  \includegraphics[width=\textwidth]{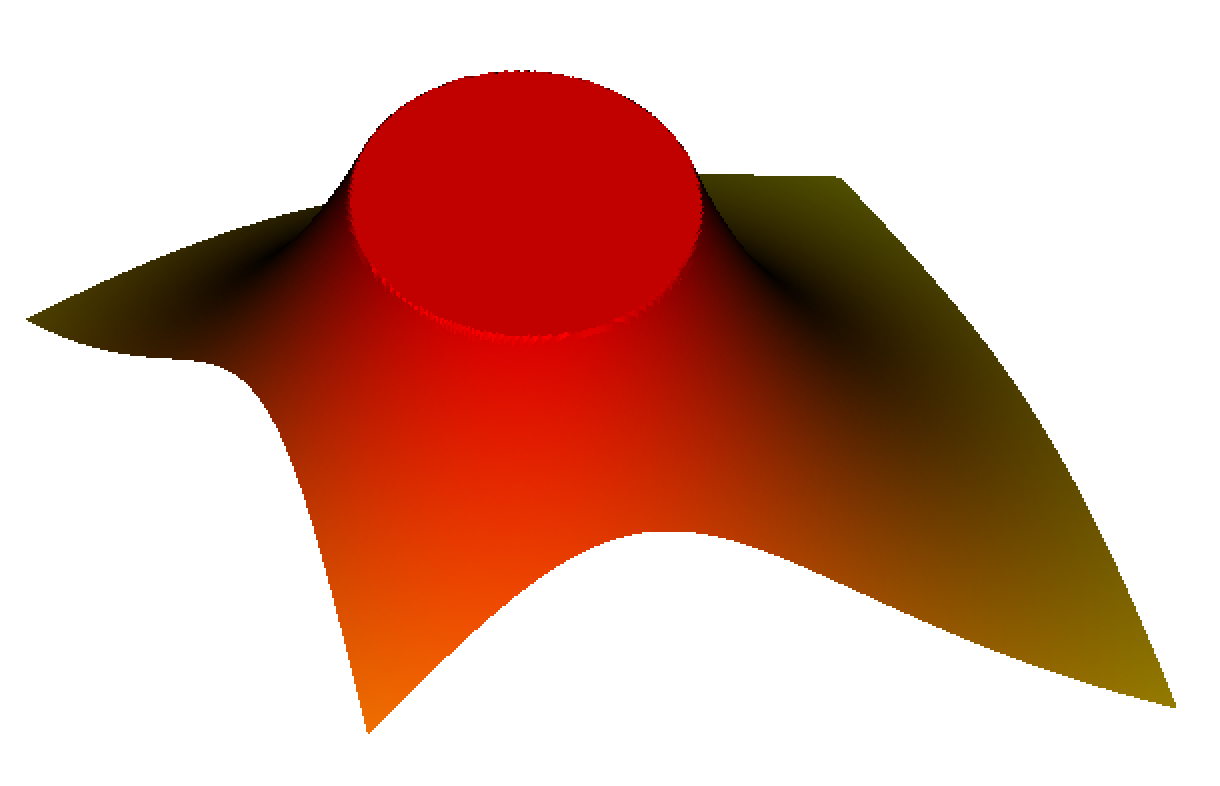}
  \caption{Elevation plot of the approximate solution to the two-dimensional model problem~\eqref{eq:example problem in 2d} in the most refined case.}
\end{figure}

The curve $\Gamma$ is a circle of radius $0.2$ with center in $c$. This function is the solution to the following problem:
\begin{equation}
\begin{aligned}
 -\Delta u &= 0 &&\text { in } \Omega\setminus\Gamma, \\
 u &= -\ln(|r|) &&\text{ on } \partial \Omega, \\
 \jump{\nu\cdot\nabla u} =f & = \frac{1}{0.2} \quad \left(=\frac{1}{|r|}=\nu\cdot\nabla u^+\right) &&\text{ on } \Gamma, \\
  \jump{ u }  & = 0 && \text{ on } \Gamma.
\end{aligned}\label{eq:example problem in 2d}
\end{equation}

We use a bi-linear finite dimensional space $W^1_h$, and show a plot of the numerical solution for $h=1/1024$ in Figure~\ref{fig:exact 2d}. We compute the error in the weighted Sobolev norms $\|\cdot\|_{0,\alpha,\Omega}$ and  $\|\cdot\|_{1,\alpha,\Omega}$ for values of $h$ varying from $1/4$ to $1/1024$, and  values of $\alpha$ varying from zero (standard Sobolev norms in $L^2(\Omega)$ and $H^1(\Omega)$) to $0.49$. 

\pgfplotstableread[comment chars={c}]{./data/error_2d.txt}\DataTwoD
\pgfplotstableread[comment chars={c}]{./data/error_3d.txt}\DataThreeD

  \begin{table}[!htb]
    \centering
    \resizebox{\textwidth}{!}{
      \TableLTwo{\DataTwoD}
    }
  \caption{Error in the weighted $L^2_\alpha$ norm $\|u-u_h\|_{0,\alpha,\Omega}$ for different values of $\alpha$ in the two-dimensional case.}
    \label{tab:convergence-rates-L2-2d}
  \end{table}

  \begin{figure}
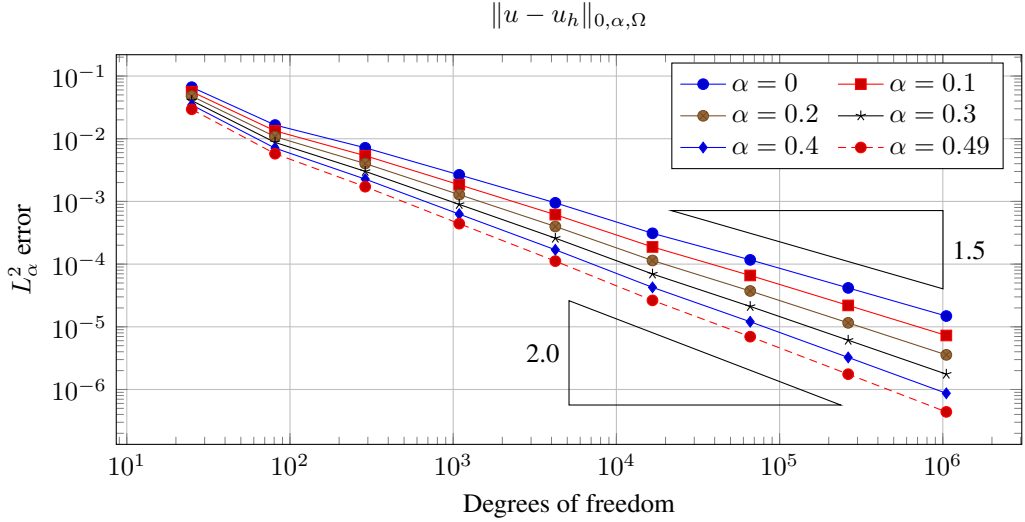

    \begin{center}
      \PlotLTwo{\DataTwoD}{}{
          \logLogSlopeTriangle{0.913}{0.3}{0.6}{.75}{black}{1.5};
          \logLogSlopeTriangleReversed{0.8}{0.3}{0.101}{1.0}{black}{2.0};
      }
    \end{center}
    \caption{Error in the weighted $L^2_\alpha$ norm $\|u-u_h\|_{0,\alpha,\Omega}$ for different values of $\alpha$ in the two-dimensional case. The black triangles show two representative rates of decrease of the error in terms of powers of the mesh size $h$.}
    \label{fig:error-L2-2d}
  \end{figure}

We report the errors in the weighted $L^2_\alpha(\Omega)$ norm in Table~\ref{tab:convergence-rates-L2-2d} and in Figure~\ref{fig:error-L2-2d}, and for the  weighted $H^1_\alpha(\Omega)$ norm in Table~\ref{tab:convergence-rates-H1-2d} and in Figure~\ref{fig:error-H1-2d}.

From the tables we verify the results of Theorem~\ref{theo:main-estimate}, and we observe rates of convergence in the standard $L^2(\Omega)$ and $H^1(\Omega)$ Sobolev norms which are coherent with the $H^{3/2}(\Omega)$ global regularity of the solution. In particular we expect a convergence rate of order $3/2$ for the $L^2(\Omega)$ norm and $1/2$ for the  $H^1(\Omega)$ norm. When increasing $\alpha$ to a value close to $1/2$, we observe that the errors in the weighed norms converge to the optimal rates (in this case two and one).

  \begin{table}[!htb]
    \centering
    \resizebox{\textwidth}{!}{
      \TableHOne{\DataTwoD}
    }
  \caption{Error in the weighted $H^1_\alpha$ norm $\|u-u_h\|_{1,\alpha,\Omega}$ for different values of $\alpha$ in the two-dimensional case.}
    \label{tab:convergence-rates-H1-2d}
  \end{table}

  \begin{figure}
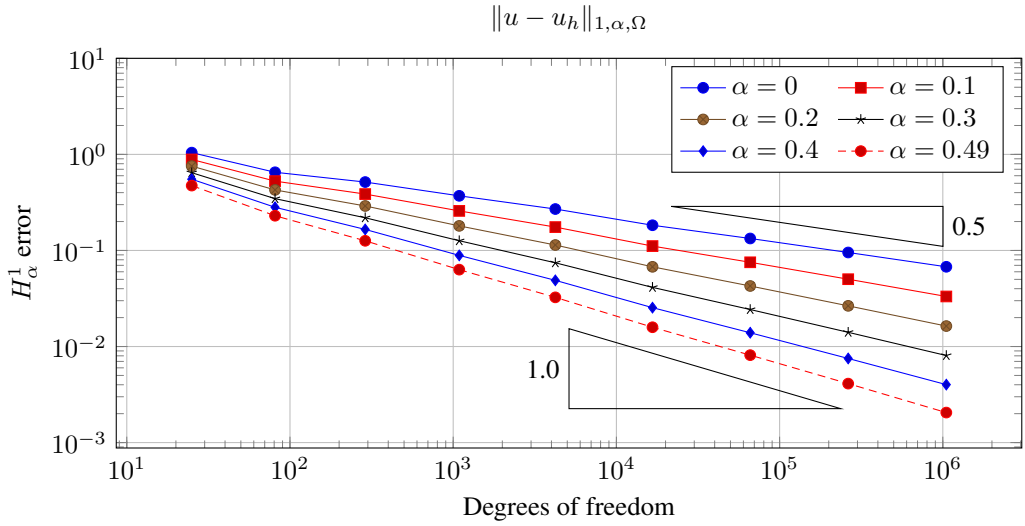

    \begin{center}
      \PlotHOne{\DataTwoD}{10}{        
          \logLogSlopeTriangle{0.913}{0.3}{0.62}{.25}{black}{0.5};
          \logLogSlopeTriangleReversed{0.8}{0.3}{0.101}{0.5}{black}{1.0};
        }
    \end{center}
    \caption{Error in the weighted $H^1_\alpha$ norm $\|u-u_h\|_{1,\alpha,\Omega}$ for different values of $\alpha$ in the two-dimensional case.  The black triangles show two representative rates of decrease of the error in terms of powers of the mesh size $h$.}
    \label{fig:error-H1-2d}
  \end{figure}

\subsection{Three-dimensional case}
\label{sec:three-dimens-case}

In the three-dimensional case, we define the exact solution to be
\begin{equation}
  \begin{aligned}
    c = (0.3,0.3,0.3) \\ 
    r := x-c \\
  \end{aligned}
  \qquad
  u = 
  \begin{cases}
    \frac{1}{|r|} & \text{ if } |r| > 0.2, \\
    \frac{1}{0.2} & \text{ if } |r| \leq 0.2.
  \end{cases}
\label{eq:exact solution 3d}
\end{equation}

\begin{figure}\centering

  \includegraphics[width=\textwidth]{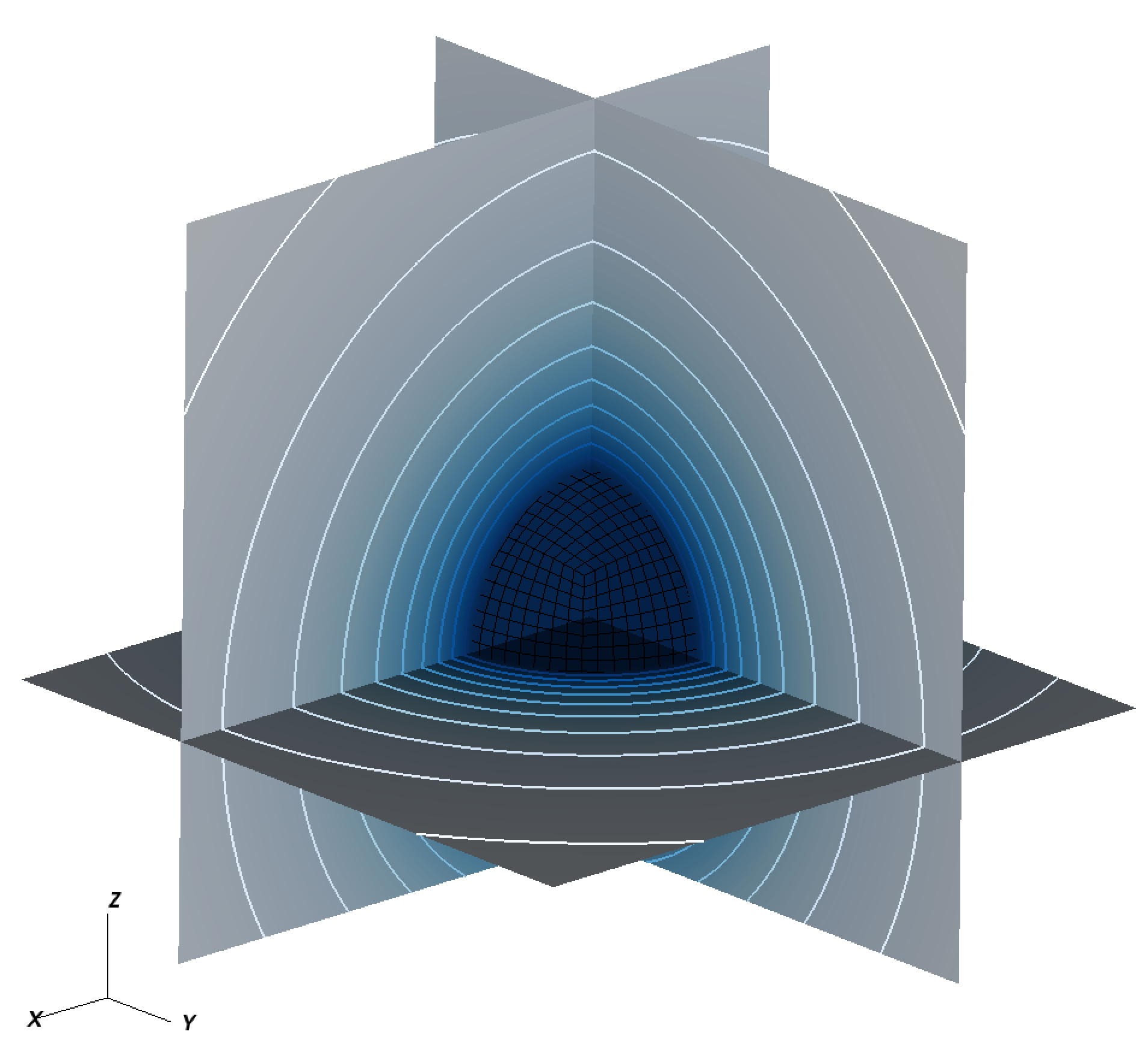}
  \caption{Sections and contour plots of the approximate solution to the three-dimensional model problem~\eqref{eq:example problem in 3d} in the most refined case.}
  \label{fig:exact 3d}
\end{figure}

The surface $\Gamma$ is a sphere of radius $0.2$ with center in $c$. This function is the solution to the following problem:
\begin{equation}
\begin{aligned}
 -\Delta u &= 0 &&\text { in } \Omega\setminus\Gamma,\\
 u &= \frac{1}{|r|} &&\text{ on } \partial \Omega,\\
 \jump{\nu\cdot\nabla u} =f & = \frac{1}{0.2^2} \quad \left(=\frac{1}{|r|^2}=\nu\cdot\nabla u^+\right) &&\text{ on } \Gamma,\\
  \jump{ u }  & = 0 && \text{ on } \Gamma.
\end{aligned}\label{eq:example problem in 3d}
\end{equation}

We use a bi-linear finite dimensional space $W^1_h$, and show a plot of the numerical solution for $h=1/128$ in Figure~\ref{fig:exact 3d}. We compute the error in the weighted Sobolev norms $\|\cdot\|_{0,\alpha,\Omega}$ and  $\|\cdot\|_{1,\alpha,\Omega}$ for values of $h$ varying from $1/4$ to $1/128$, and  values of $\alpha$ varying from zero (standard Sobolev norms in $L^2(\Omega)$ and $H^1(\Omega)$) to $0.49$. 

  \begin{table}[!htb]
    \centering
    \resizebox{\textwidth}{!}{
      \TableLTwo{\DataThreeD}
    }
  \caption{Error in the weighted $L^2_\alpha$ norm $\|u-u_h\|_{0,\alpha,\Omega}$ for different values of $\alpha$ in the three-dimensional case.}
    \label{tab:convergence-rates-L2-3d}
  \end{table}

  \begin{figure}
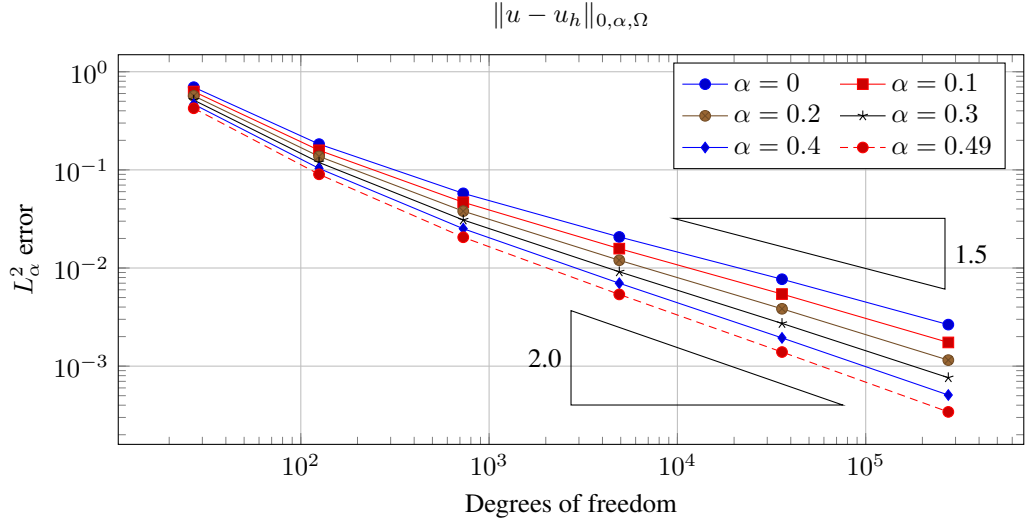

    \begin{center}
      \PlotLTwo{\DataThreeD}{}{   
          \logLogSlopeTriangle{0.913}{0.3}{0.58}{1.5/3}{black}{1.5};
          \logLogSlopeTriangleReversed{0.8}{0.3}{0.101}{2.0/3.0}{black}{2.0};
        }
    \end{center}
    \caption{Error in the weighted $L^2_\alpha$ norm $\|u-u_h\|_{0,\alpha,\Omega}$ for different values of $\alpha$ in the three-dimensional case.  The black triangles show two representative rates of decrease of the error in terms of powers of the mesh size $h$.}
    \label{fig:error-L2-3d}
  \end{figure}

We report the errors in the weighted $L^2_\alpha(\Omega)$ norm in Table~\ref{tab:convergence-rates-L2-3d} and in Figure~\ref{fig:error-L2-3d}, and for the  weighted $H^1_\alpha(\Omega)$ norm in Table~\ref{tab:convergence-rates-H1-3d} and in Figure~\ref{fig:error-H1-3d}.

From the tables we verify again the results of Theorem~\ref{theo:main-estimate} in the three dimensional case. We observe rates of convergence in the standard $L^2(\Omega)$ and $H^1(\Omega)$ Sobolev norms which are coherent with the $H^{3/2}(\Omega)$ global regularity of the solution. In particular we expect a convergence rate of order $3/2$ for the $L^2(\Omega)$ norm and $1/2$ for the  $H^1(\Omega)$ norm. When increasing $\alpha$ to a value close to $1/2$, we observe that the errors in the weighed norms converge to the optimal rates (in this case two and one).

  \begin{table}[!htb]
    \centering
    \resizebox{\textwidth}{!}{
      \TableHOne{\DataThreeD}
    }
  \caption{Error in the weighted $H^1_\alpha$ norm $\|u-u_h\|_{1,\alpha,\Omega}$ for different values of $\alpha$ in the three-dimensional case.}
    \label{tab:convergence-rates-H1-3d}
  \end{table}

  \begin{figure}
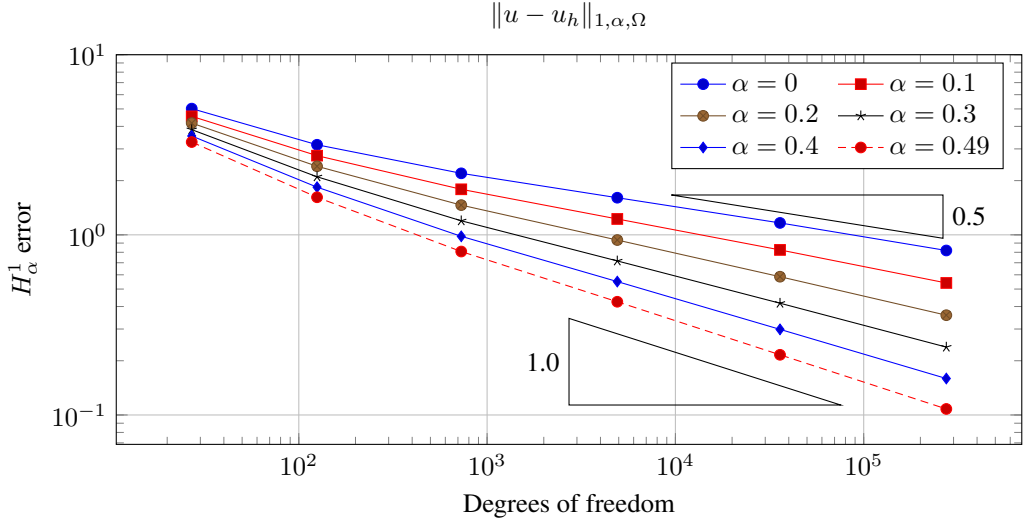

    \begin{center}
      \PlotHOne{\DataThreeD}{10}{
          \logLogSlopeTriangle{0.913}{0.3}{0.64}{.5/3.0}{black}{0.5};
          \logLogSlopeTriangleReversed{0.8}{0.3}{0.101}{1.0/3.0}{black}{1.0};
        }
    \end{center}
    \caption{Error in the weighted $H^1_\alpha$ norm $\|u-u_h\|_{1,\alpha,\Omega}$ for different values of $\alpha$ in the three-dimensional case.  The black triangles show two representative rates of decrease of the error in terms of powers of the mesh size $h$.}
    \label{fig:error-H1-3d}
  \end{figure}

\section{Conclusions}
\label{sec:conclusions}

One of the major point against the use of immersed boundary methods and their variational counterparts, comes from the unfavourable comparison in convergence rates that can be achieved using matching grid methods (ALE~\cite{Hirt1974,DoneaGiulianiHalleux-1982-a}), or enriching techniques (IIM~\cite{Leveque1994}, X-FEM~\cite{Mittal2005b}). 

In this work we have shown that this detrimental effect on the convergence properties is only a local phenomena, restricted to a small neighbourhood of the interface. In particular we have proved that optimal approximations can be constructed in a natural and inexpensive way, simply by reformulating the problem in a distributionally consistent way, and by resorting to weighted norms  when computing the global error of the approximation, where the weight is an appropriate power of the distance from the interface.
Weighted Sobolev spaces~\cite{Kufner1985,Turesson2000}, provide a natural framework for the study of the convergence properties of problems with singular sources~\cite{Agnelli2014} or problems with singularities in the domain~\cite{Belhachmi2006,Duran2009a}.

The method we have presented has the great advantage of not requiring
any change in the numerical approximation scheme, which is maintained
the same as if no interface were present, requiring only the
construction of a special right hand side, incorporating the jump
conditions in a simple singularity~\cite{Peskin2002,
  BoffiGastaldiHeltaiPeskin-2008-a}. The analysis is based on results
from~\cite{Drelichman2018}, while the numerical techniques borrow
heavily from the formalism and the results presented
in~\cite{Dangelo2012,DAngelo2008}.

Applications of this discretisation technique to fluid structure
interaction problems is well known and dates back to the early
seventies~\cite{Peskin1972}. Recent developments in finite element
variants are available, for example in~\cite{Heltai2012b, Roy2015},
and similar constructions are used to impose boundary conditions on
ocean circulation simulations~\cite{Rotundo2016}. Applications of this
technique could be used, for example, to allow more general classes of
doping profiles in doping optimization problems for semi-conductor
devices~\cite{PeschkaRotundoThomas2016,Peschka2018}.

It is still unclear if the same techniques may be used for the treatment of jumps in the solution itself, as in these cases the regularity gain that could be achieved by weighted Sobolev spaces alone may not be sufficient, and we are currently exploring alternative approximation frameworks, following the lines of~\cite{Nochetto2016}.

\section*{Acknowledgments}

N.R. acknowledges support by DFG via SFB 787 Semiconductor Nanophotonics: Materials, Models, Devices, project B4 ``Multi-dimensional Modeling and Simulation of electrically pumped semiconductor-based Emitters''.

\bibliographystyle{abbrv}
\bibliography{ifem-luca-nella}


\end{document}